\documentclass[review,11pt]{elsarticle}

\usepackage{amssymb}

\usepackage{amsmath}
\usepackage{multirow}
\usepackage{diagbox}
\usepackage{subfigure}
\usepackage{graphicx}
\usepackage{float}
\usepackage{array}
\usepackage{amsthm}

\usepackage{makecell}

\theoremstyle{definition}

\usepackage{amsmath,amssymb}

\def\tsc#1{\csdef{#1}{\textsc{\lowercase{#1}}\xspace}}
\tsc{WGM}
\tsc{QE}
\tsc{EP}
\tsc{PMS}
\tsc{BEC}
\tsc{DE}

\newtheorem{theorem}{Theorem}
\newtheorem{lemma}{Lemma}
\newtheorem{proposition}{Proposition}

\newtheorem{assumption}{Assumption}
\newproof{pf}{Proof}

\usepackage[framemethod=tikz]{mdframed}
\usetikzlibrary{shadows}

\newmdenv[shadow=false,shadowcolor=black,font=\sffamily,rightmargin=0.1pt]{shadedbox}

\newcommand{\F}{\mathcal{F}}
\newcommand{\Prob}{\mathbb{P}}
\newcommand{\Ex}{\mathbb{E}} 
 
\newcommand{\R}{\mathbb{R}}
\newcommand{\bz}{\mathbf{z}}
\newcommand{\bx}{\mathbf{x}}
\newcommand{\by}{\boldsymbol{y}}

\newcommand{\V}{\mathbb{V}}
\newcommand{\bX}{\boldsymbol{X}}

\usepackage{geometry}
\journal{\;\;}

\begin{document}
\begin{frontmatter}




\title{Distributionally Robust Optimization under Mean-Covariance Ambiguity Set and Half-Space Support for Bivariate Problems}



\author[inst1]{Jiayi Guo}
\ead{guo.jiayi@sufe.edu.cn}

\author[inst1]{Hao Qiu}
\ead{qiu.hao@163.sufe.edu.cn}

\author[inst2,inst3]{Zhen Wang}
\ead{wangzhen@cuhk.edu.cn}
\author[inst2]{Zizhuo Wang  \corref{cor1}}
\ead{wangzizhuo@cuhk.edu.cn}
\author[inst1]{Xinxin Zhang}
\ead{xinxin_zhang@163.sufe.edu.cn}

\cortext[cor1]{Corresponding author}

\affiliation[inst1]{organization={School of Information Management and Engineering, Shanghai University of Finance and Economics},
            city={Shanghai},
            postcode={200433}, 
            country={P.R. China}
           }


\affiliation[inst2]{organization={School of Data Science, The Chinese University of Hong Kong, Shen Zhen},
            city={Shenzhen, Guangdong},
            postcode={518172}, 
            country={P.R. China}
            }

\affiliation[inst3]{organization={University of Science and Technology of China},
            city={Hefei, Anhui},
            postcode={230026}, 
            country={P.R. China}
            }

\begin{abstract}
In this paper, we study a bivariate distributionally robust optimization problem with mean-covariance ambiguity set and half-space support. Under a conventional type of objective function widely adopted in inventory management, option pricing, and portfolio selection, we obtain closed-form tight bounds of the inner problem in six different cases. Through a primal-dual approach, we identify the optimal distributions 
in each case. As an application in inventory control, we first derive the optimal order quantity and the corresponding worst-case distribution, extending the existing results in the literature.
Moreover, we show that under the distributionally robust setting, a centralized inventory system does not necessarily reduce the optimal total inventory, which contradicts conventional wisdom.
Furthermore, we identify two effects, a conventional pooling effect, and a novel shifting effect, the combination of which determines the benefit of incorporating the covariance information in the ambiguity set. Finally, we demonstrate through numerical experiments the importance of keeping the covariance information in the ambiguity set instead of compressing the information into one dimension.
\end{abstract}




\end{frontmatter}


\section{Introduction}

In recent years, {\it distributionally robust optimization} (DRO) has become a popular approach to address optimization problems affected by uncertainty. Its applications have received considerable attention in various fields including economics, management science, and mathematical finance \cite{bertsimas2019adaptive}. The appeal of distributionally robust optimization lies in its flexibility in specifying uncertainty beyond a fixed probability distribution, as well as in its ability to produce computationally tractable models \cite{chen2019distributionally}.
We refer the readers to \cite{scarf1958min,dupavcova1987minimax,gilboa2004maxmin,breton1995algorithms,shapiro2002minimax} for some comprehensive reviews for DRO problems.

The problem studied in distributionally robust optimization is as follows: 
\begin{equation}
\inf_{\boldsymbol{y}\in Y}\ \sup_{\Prob\in \F}\ \Ex_{\Prob} \left[f\left(\by,\bX \right)\right]
   \label{generalmodel}
\end{equation}
where $X\in \mathbb{R}^n$ is the realization of the uncertainty, $\by \in Y \subseteq \mathbb{R}^m$ is the decision taken before the realization, and  $f\left(\by, \bX \right): \R^m\times\R^n \rightarrow \R$ is the objective function. Given the decision $\boldsymbol{y}$, the {\it inner problem} of \eqref{generalmodel} evaluates the worse-case objective:
\begin{equation}
	\centering
	\sup_{\Prob\in \mathcal{F}}\ \Ex_{\Prob} \left[f\left(\by,\bX \right)\right],
	\label{innerproblem}
\end{equation}
where the {\it ambiguity set}, denoted by $\F$, characterizes a set of possible distributions $X$ may follow. In the literature, various characterizations of ambiguity set have been considered, such as support \cite{ghaoui2003worst, natarajan2010tractable}, moments \cite{scarf1958min, bertsimas2005optimal}, shape \cite{popescu2005semidefinite}, and dispersions \cite{hanasusanto2017ambiguous}.
Among them, one of the most classical ambiguity sets is the {\it mean-covariance} ambiguity set given as follows:
\begin{equation}
	\centering
	\mathcal{F} = \left\{\mathbb{P}\in \mathbb{M}(\mathbb{S}) : \mathbb{E}_{\mathbb{P}} [\boldsymbol{X}] = \boldsymbol{\mu}, \mathbb{E}_{\mathbb{P}} \left[\boldsymbol{X}\boldsymbol{X}^T \right] = \Sigma \right\},
	\label{mean-var}
\end{equation}
where $\boldsymbol{\mu}$, $\Sigma$ are the mean and covariance of the distribution and $\mathbb{S}$ is the support.
If $\mathbb{S}=\R_+^n$ or $\mathbb{S}=\R^n$, then we call the support of $\F$ {\it half-space} or {\it full-space} respectively. The key advantage of the mean-covariance ambiguity set lies in its simplicity and computational tractability, especially under the full-space support.
Generally speaking, the computational complexity of DRO problem depends on the complexity of solving the inner problem, which is a semi-infinite linear program on its probability measure, and NP-hard in general \cite{bertsimas2005optimal}.

In this paper, we consider a class of widely-adopted loss functions for the inner problem as follows:
\begin{equation}
	\centering
	\ell (\boldsymbol{X})=\max\{u_1\boldsymbol{w}^T{\boldsymbol{X}}+v_1, u_2\boldsymbol{w}^T{\boldsymbol{X}}+v_2\},
	\label{obj}
\end{equation}
with $ \boldsymbol{X} \in \mathbb{R}^{n}$, $\boldsymbol{w}\in \mathbb{R}_+^n$ and $u_1, u_2, v_1, v_2\in\mathbb{R}$. 
Such a loss function is adopted in a wide range of applications,
including inventory management \cite{scarf1958min,doan2020robust}, option pricing \cite{natarajan2010tractable,tian2008moment, zuluaga2009third, bertsimas2002relation}, and portfolio selection \cite{doan2015robustness,mohajerin2018data,liu2022kernel,zhu2009worst}. Regarding the ambiguity set, 
we focus on
the mean-covariance ambiguity set, which is a common choice and has been studied in every application listed above (see, for instance, \cite{natarajan2010tractable,chen2011tight,scarf1958min,bertsimas2002relation,tian2008moment}). Furthermore, we consider the problem under the half-space support, which is of relevance under many situations  (e.g., nonnegative demand in inventory control \cite{govindarajan2021distribution}, nonnegative losses in a stop-loss contract \cite{tian2008moment}). Moreover, by incorporating the half-space information, the ambiguity sets shrink, resulting in a less conservative optimal solution.

However, for problems with a mean-covariance ambiguity set, the methodology for solving the DRO problem under half-space support is less well-studied than that under full-space support. Specifically, when considering the problems with objectives of a linear reward function under the full-space support, Popescu demonstrates a powerful projection property that can translate an $n-$dimensional multivariate inner problem into an equivalent univariate problem \cite{popescu2007robust}. 
Unfortunately, it is not easy 
to incorporate half-space support information in such an approach. 
In fact, as shown in \cite{bertsimas2002relation}, 
even determining the feasibility of such a problem under half-space support is NP-hard.
Therefore, given mean and covariance information, much research on incorporating half-space support into DRO  focus on approximation algorithms. For example, Rujeerapaiboon et al. \cite{rujeerapaiboon2018chebyshev} study the upper Chebyshev bound of a product of non-negative random variables given their first two  
moments. They show that in order to obtain a tractable numerical procedure, the first two moments should follow a permutation-symmetric structure. 
Kong et al. \cite{kong2013scheduling} formulate a semidefinite programming relaxation for the appointment scheduling problem given the mean and covariance information for the nonnegative 
random service time variables.
Natarajan et al. \cite{natarajan2010tractable} provide a mathematically tractable lower bound for the expected piecewise linear utility function 
by taking a convolution of two problems that are computationally simple through relaxing either the half-space support to the full-space or the mean-covariance ambiguity set to the one with only mean information. 


In this paper, we further shed light on the DRO problem under mean-covariance ambiguity set and half-space support. Particularly, we are able to derive an analytical solution to such a problem for the two-dimensional case.
Two-dimensional problem 
has been extensively studied in the literature, as it is simple
and useful to illustrate the relation between the optimal decisions and the correlations of two random variables. 
For example, two-dimensional problems are considered in the variations of newsvendor model including dual sourcing \cite{tomlin2005value}, two markets \cite{kouvelis1997newsvendor}, and two products \cite{lau1988maximizing,li1991two}. Moreover, the stop-loss problem in option pricing typically involves two types of losses (i.e., property losses and liability losses) \cite{tian2008moment}. As far as we know, it is computationally challenging even to solve two-dimensional DRO problems with mean-covariance ambiguity set under half-space support.
Specifically, a typical approach to solving those problems under half-space support is to relax the support to the full-space \cite{govindarajan2021distribution,cox2010bounds,bertsimas2002relation}. 
Moreover, as the projection theorem \cite{popescu2007robust} cannot be applied to those problems, both works \cite{cox2010bounds,bertsimas2002relation} obtain numerical solutions to the relaxed problems. Govindarajan et al. achieve analytical solutions to the relaxed problem under strong assumptions \cite{govindarajan2021distribution}. 
Unlike the relaxation of the support, Tian relaxes the two-dimensional problem into a univariate one and applies the moment approach \cite{tian2008moment}. To sum up, the analytical result of a two-dimensional problem with mean-covariance ambiguity set under half-space support is largely missing in the literature, and our work fills in the gap with a conclusive answer.

We summarize the contributions of this paper as follows:
\begin{itemize}
	\item[1.]  We propose an analytical solution for the two-dimensional DRO problem with mean-covariance ambiguity set and loss function $\ell$ under half-space support. Specifically, the optimal value of this problem can be characterized by six different cases. To obtain the optimal solution in each case, we extend the primal-dual approach in \cite{guo2022unified}, designed for the univariate moment problems, to our two-dimensional problems. 
 	\item[2.]  We extend the loss function $\ell$ defined in (\ref{obj}) to a generalized multi-piece quadratic objective function and provide a semi-definite programming reformulation for the DRO problem with mean-covariance ambiguity set and the generalized loss function. 
	\item[3.]  We apply our analytical result to inventory control problems, providing the optimal order quantity and the worst-case distribution, which is an extension of the result in \cite{scarf1958min}. 
 Moreover, we find that under the DRO setting with mean-covariance ambiguity set, inventory pooling does not necessarily reduce the optimal total inventory, which contradicts with the conventional wisdom that pooling can reduce the total inventory. Furthermore, we identify two effects, a conventional pooling effect and a novel shifting effect, which together determine the benefit of incorporating the covariance information in the ambiguity set. Finally, we demonstrate through numerical experiments the importance of keeping the covariance information in the ambiguity set, instead of compressing the information in one dimension.
\end{itemize}

The rest of this paper is organized as follows. In Section \ref{section2}, we introduce our model.
 In Section \ref{section3}, we formally present the bivariate moment problem and its closed-form solution. We also demonstrate some properties of the worst-case distribution and provide a numerical approach for the generalized objectives. In Section \ref{section: Impact of Pooling}, we study an inventory control problem as the outer problem and provide some managerial insights for adopting the mean-covariance DRO model for such problems compared to other approaches. We conclude the paper and point out some future research directions in Section \ref{sec:conclusion}.
 
%

\section{Model}\label{section2}

Consider the following distributionally robust optimization problem
\begin{equation}
\inf_{\boldsymbol{y}\in Y}\ \sup_{\mathbb{P}\in \mathcal{F}}\ \mathbb{E}_{\mathbb{P}} \left[f\left(\boldsymbol{X}, \boldsymbol{y} \right)\right]
\end{equation}
with feasible region $Y\subseteq \mathbb{R}^m$, ambiguity set $\mathcal{F}$, and loss function $f: \mathbb{R}^n\times \mathbb{R}^m \rightarrow \mathbb{R}$. The loss function depends both on the decision vector $\boldsymbol{y}\in \mathbb{R}^m$ and the random vector $\boldsymbol{X}\in \mathbb{R}^n$.

We first investigate the inner worst-case expectation problem. For ease of notation, we suppress the dependence on the decision variable $\boldsymbol{y}$. Specifically, we consider the inner worst-case expectation problem of the form
\begin{equation}
\label{eq:innerWorstCaseExpectationProblem}
\sup_{\mathbb{P}\in \mathcal{F}}\ \mathbb{E}_{\mathbb{P}} \left[\ell\left( \boldsymbol{w}^T \boldsymbol{X} \right)\right]
\end{equation}
where the loss function $\ell:\mathbb{R}\rightarrow \mathbb{R}$ is a convex piecewise linear function with the following form
$$
\ell (x) = \max\left\{u_1x +v_1, u_2x +v_2 \right\},
$$
and $\boldsymbol{w}\in \mathbb{R}_+^n$ is a nonnegative coefficient of the random vector $\boldsymbol{X}$. The ambiguity set $\mathcal{F}$ of the distribution $\mathbb{P}$ consists of nonnegative random vector $\boldsymbol{X}$ with given mean $\boldsymbol{\mu}$ and second-order moment $\Sigma$, that is
\begin{equation}
\label{eq:multidimensionAmbiguitySet}
\mathcal{F} = \left\{\mathbb{P}\in \mathbb{M}(\mathbb{R}_+^n) : \mathbb{E}_{\mathbb{P}} [\boldsymbol{X}] = \boldsymbol{\mu}, \mathbb{E}_{\mathbb{P}} \left[\boldsymbol{X}\boldsymbol{X}^T \right] = \Sigma \right\}.
\end{equation}
The problem of form (\ref{eq:innerWorstCaseExpectationProblem}) is widely adopted in DRO. We provide the following two applications as examples.

\paragraph{Example 1. (Multi-dimensional Newsvendor Problem)} We consider a multi-dimensional distributionally robust 
newsvendor problem in a centralized inventory management setting (see \cite{bimpikis2016inventory}) as follows:
\begin{equation}
    \label{eq:dronewsvender}
    \inf_{q\geq 0} \left\{\sup_{\mathbb{P}\in \mathcal{F}} \mathbb{E}_{\mathbb{P}} \Big[\Big(\sum_{i=1}^n X_i -q \Big)_+\Big] + (1-\eta)q \right\},
\end{equation}
where we denote $(\cdot)_+=\max\{\cdot, 0\}$ in the rest of this paper. Specifically, 
$q$ and $\eta$ are the order quantity and the critical ratio respectively, and $\sum_{i=1}^n X_i $ represents the pooling of uncertain demands.
Obviously, the worst-case expectation of the inner problem can be considered as a special case of problem
(\ref{eq:innerWorstCaseExpectationProblem}) with $u_1 = 1$, $v_1 = -q$, $u_2 = v_2 = 0$ and $\boldsymbol{w} = \boldsymbol{1}$.

\paragraph{Example 2. (Mean-CVaR Portfolio Selection Problem)}
We consider a capital market consisting of $n$ assets whose uncertain returns are captured
by an $n-$dimensional random vector $\boldsymbol{X}$. A portfolio is encoded by an  $n-$dimensional vector $\boldsymbol{w}$ which has to satisfy some constraints $\boldsymbol{w}\in W$.
The distributionally robust mean-CVaR portfolio selection problem can be formulated as follows:
$$\inf_{\boldsymbol{w}\in W} \left\{\sup_{\mathbb{P}\in \mathcal{F}} \Big\{\mathbb{E}_{\mathbb{P}} \left[-\boldsymbol{w}^T\boldsymbol{X} \right] + \rho\ \mathbb{P}\text{-CVaR}_{\alpha}\left(-\boldsymbol{w}^T\boldsymbol{X} \right) \Big\} \right\}$$
or equivalently,
$$\inf_{\boldsymbol{w}\in W} \bigg\{-\boldsymbol{w}^T\boldsymbol{\mu} + \rho\  \inf_{\tau\in \mathbb{R}}\Big\{  \tau + \frac{1}{\alpha} \sup_{\mathbb{P}\in \mathcal{F}}\mathbb{E}_{\mathbb{P}} \left[ \left(-\boldsymbol{w}^T\boldsymbol{X}-\tau \right)_+ \right] \Big\} \bigg\}$$
where the equivalence is formally proved in
\cite{rockafellar2000optimization}.
Again, the worst-case expectation of the inner problem can be considered as a special case 
of problem (\ref{eq:innerWorstCaseExpectationProblem}) with 
$u_1 = -1$, $v_1 = -\tau$, and $u_2 =v_2 = 0$.

Without loss of generality, we assume $u_1<u_2$ in the rest of this paper. Otherwise, supposing $u_1=u_2=0$, the loss function $\ell(x)$ can be reduced to a linear function 
$\ell (x) = ux + \max\{v_1, v_2\}$. As a result, problem (\ref{eq:innerWorstCaseExpectationProblem}) has a trivial solution as
\begin{equation*}
     \mathbb{E}_{\mathbb{P}} [\ell\left( \boldsymbol{w}^T \boldsymbol{X} \right)]=u_1 \mathbb{E}_{\mathbb{P}} [ \boldsymbol{w}^T \boldsymbol{X} ] + \max\{v_1, v_2\} = u_1 \boldsymbol{w}^T \boldsymbol{\mu} + \max\{v_1, v_2\}.
\end{equation*}

 In fact, with this assumption, we can simplify the formulation of problem (\ref{eq:innerWorstCaseExpectationProblem})
by the following proposition.

\begin{proposition}
    \label{proposition1}
Suppose $u_1<u_2$. We can reformulate 
the problem $\sup_{\mathbb{P}\in \mathcal{F}}\ \mathbb{E}_{\mathbb{P}} \left[\ell\left( \boldsymbol{w}^T \boldsymbol{X} \right)\right]$ in (\ref{eq:innerWorstCaseExpectationProblem}) as follows:
$$
(u_2-u_1) \sup_{\mathbb{P}\in \tilde{\mathcal{F}}} \mathbb{E}_{\mathbb{P}} \left[\left(\boldsymbol{\boldsymbol{1}}^T \tilde{\boldsymbol{X}} + \frac{v_2-v_1}{u_2-u_1} \right)_+ \right] + u_1 \boldsymbol{w}^T\boldsymbol{\mu} +v_1
$$
with the transformed random variable $\tilde{\boldsymbol{X}} = \boldsymbol{w}\circ \boldsymbol{X}$ and the transformed ambiguity set
$$
	\tilde{\mathcal{F}} = \left\{\mathbb{P}\in \mathbb{M}(\mathbb{R}_+^n) : \mathbb{E}_{\mathbb{P}} [\tilde{\boldsymbol{X}}] = \boldsymbol{w}\circ \boldsymbol{\mu}, \mathbb{E}_\mathbb{P} \big[\tilde{\boldsymbol{X}}\tilde{\boldsymbol{X}}^T \big] = \boldsymbol{w}\boldsymbol{w}^T\circ \Sigma \right\}
$$
where the notation $\circ$ denotes the Hadamard product. 
\end{proposition}

This proposition demonstrates that it is sufficient to solve the simplified problem in the form of
\begin{equation}
\label{eq:innerNewsvendor}
\sup_{\mathbb{P}\in \mathcal{F}}\ \mathbb{E}_{\mathbb{P}} \Big[\Big(\sum_{i=1}^n X_i -q \Big)_+\Big]
\end{equation}
so as to solve the original problem (\ref{eq:innerWorstCaseExpectationProblem}) with $q=-\frac{v_2-v_1}{u_2-u_1}$. We relegate the proof of Proposition \ref{proposition1} to Appendix A.

Generally speaking, even determining the feasibility of problem (\ref{eq:innerNewsvendor}) is NP-hard as shown in 
\cite{bertsimas2002relation}. For univariate case ($n=1$), Scarf derives a closed-form solution in  \cite{scarf1958min}, which we state below for completeness.
 \begin{lemma}[\cite{scarf1958min}]
Let the ambiguity set 
$\mathcal{F}^\dag=\left\{\mathbb{P}\in \mathbb{M}(\mathbb{R}_+) : \mathbb{E}_{\mathbb{P}} [X] = \mu, \mathbb{E}_{\mathbb{P}} [X^2] = \Sigma \right\}$.
\begin{enumerate}
    \item[1.] If $0 \leq q \leq \frac{\Sigma}{2\mu}$, then the optimal value is
\begin{equation} \label{eq: scarfOptQSm}
\sup_{\mathbb{P}\in \mathcal{F}^\dag} \mathbb{E}_{\mathbb{P}} [(X -q)_+]=\mu-q \cdot \frac{\mu^2}{\Sigma}
\end{equation}
with an optimal distribution 
$X^*=\begin{cases}
0 &\text{w.p.} \;\; 1 -  \frac{\mu^2}{\Sigma}\\
\frac{\Sigma}{\mu} &\text{w.p.}\;\; \frac{\mu^2}{\Sigma}
\end{cases}$.
\item[2.]If $q>\frac{\Sigma}{2\mu}$, then the optimal value is
\begin{equation} \label{eq: scarfOptQLg}
\sup_{\mathbb{P}\in \mathcal{F}^\dag} \mathbb{E}_{\mathbb{P}} [(X -q)_+]=\frac{1}{2}(Q-q+\mu)
\end{equation}
with an optimal distribution
$\mathbb{P}^*=\begin{cases}
q-Q &\text{w.p.} \;\; \frac{1}{2}+\frac{q-\mu}{2Q}\\
q+Q &\text{w.p.}\;\; \frac{1}{2}-\frac{q-\mu}{2Q}
\end{cases}$,
where $Q=\sqrt{q^2-2\mu q+\Sigma}$.
\end{enumerate}
\label{scarflemma}
 \end{lemma}
 
 In this paper, we study the bivariate case  ($n=2$) of problem (\ref{eq:innerNewsvendor}), which takes the correlation between two random variables into account. We provide a closed-form solution for the inner problem (\ref{eq:innerNewsvendor}), based on which an efficient algorithm for the outer DRO problem is also provided.

\section{Bivariate Moment Problem}\label{section3}

In this section, we study the bivariate case ($n=2$) of  problem (\ref{eq:innerNewsvendor}) stated as follows: 
\begin{equation}
\label{eq:bivariatePrimal}
v_P(\theta;q)=\sup_{\mathbb{P}\in \mathcal{F}(\theta)}\ \mathbb{E}_{\mathbb{P}} \left[\left(X_1+X_2 -q \right)_+\right]
\end{equation}
where $\theta=(\mu_1,\mu_2,\Sigma_{11},\Sigma_{22},\Sigma_{12})$ represents the moment information of random variables $X_1$ and $X_2$. For simplicity, we denote $(\Sigma_{11},\Sigma_{22},\Sigma_{12}):=(a\mu_1^2, b\mu_2^2, c\mu_1\mu_2)$ for the rest of this paper, and consider the ambiguity set 
\begin{equation} \label{eq: amSet}
\F(\theta)={
\left\{
\begin{array}{l}
    \mathbb{P}\in 
     \mathbb{M}(\mathbb{R}^2_+)  
\end{array} 
\left|
\begin{array}{l}
\mathbb{E}_{\mathbb{P}} [X_1] = \mu_1, \;\;\;\;\;\,
 \mathbb{E}_{\mathbb{P}} [X_2] = \mu_2   \\
\mathbb{E}_{\mathbb{P}} \left[X_1^2 \right] = a\mu_1^2, \;\;
 \mathbb{E}_{\mathbb{P}} \left[X_2^2 \right] = b \mu_2^2, \;\;
\mathbb{E}_{\mathbb{P}} \left[X_1X_2 \right] = c \mu_1 \mu_2
\end{array}\right.
\right\}
}.
\end{equation}

Thus, we have the correlation $\rho= \frac{c-1}{\sqrt{(a-1)(b-1)}}$ and the covariance matrix 
\begin{equation} \label{eq: M}
\begin{aligned}
M=\left[
\begin{array}{cc}
(a-1)\mu_1^2&(c-1)\mu_1\mu_2\\
(c-1)\mu_1\mu_2&(b-1)\mu_2^2
\end{array}
\right].
\end{aligned}
\end{equation} 
The nonemptyness of ambiguity set $\F(\theta)$ requires that the covariance matrix $M$ be positive semi-definite, which means $a\geq1$, $b \geq 1$ and $(a-1)(b-1)\geq(c-1)^2$. 
Note that $X_1,  X_2\in \mathbb{R}_+$ and thus $c \geq 0$. To exclude trivial cases, we focus on $a, b>1$.
In all, without loss of generality, we assume that our input parameters satisfy the following assumption in the remaining part of this paper.



\begin{assumption}
    \label{pro:1}
    We assume $a > 1$, $b > 1$, $c\geq 0$  and $(a-1)(b-1)\geq (c-1)^2$.
\end{assumption}
%


Next, we present the closed-form solution for the bivariate moment problem (\ref{eq:bivariatePrimal}).

\subsection{Closed-Form Solution}\label{section: Closed-form Solution}

Before establishing our closed-form solution, we first define three terms $Q_a, Q_b, Q_c$ that will be frequently used in this subsection:
$$
\begin{aligned}
Q_a &= \sqrt{q^2 - 2q \frac{a-c}{a-1}\mu_2 + \frac{ab-c^2}{a-1}\mu_2^2}, \\
Q_b &= \sqrt{q^2 - 2q \frac{b-c}{b-1}\mu_1 + \frac{ab-c^2}{b-1}\mu_1^2}, \\
Q_c &= \sqrt{q^2 - 2q(\mu_1 + \mu_2) + a\mu_1^2 + b\mu_2^2 + 2c\mu_1\mu_2}.
\end{aligned}
$$
In the following lemma, we present the conditions that correspond to six different optimal values of problem (\ref{eq:bivariatePrimal}). Specifically, 
this lemma shows that such six conditions divide the feasible region into six disjoint sets.

\begin{lemma}
\label{lem:conditions}
	Every feasible input of $(\theta,q)$ satisfies one and only one of the six conditions in the following table.
 \begin{center}
        \begin{tabular}{|c|c|c|c|c|c|c}
            \hline
             Condition 1 & Condition 2 & Condition 3 & Condition 4 & Condition 5 & Condition 6 \\
            \hline
              $Q_a \geq q,$ 
              & $Q_b < q,$
              & $Q_a < q,$  & $Q_b < q,$ & $Q_a < q,$ & $Q_a > |\zeta_a|,$ \\
              $Q_b \geq q$ & $Q_b \leq \zeta_b$ & $Q_a \leq \zeta_a$ & $Q_b \leq -\zeta_b$ & $Q_a \leq -\zeta_a$ & $Q_b > |\zeta_b|$ \\
             \hline
        \end{tabular}
\end{center}
where $\zeta_a=a\mu_1+c\mu_2-q$ and $\zeta_b=c\mu_1+b\mu_2-q$.
\end{lemma}


We relegate the proof of this lemma to Appendix B.1. Below we provide a sample plot as a graphical illustration. 
\begin{figure}[htbp]
	\centering
	\includegraphics[width=0.55\linewidth]{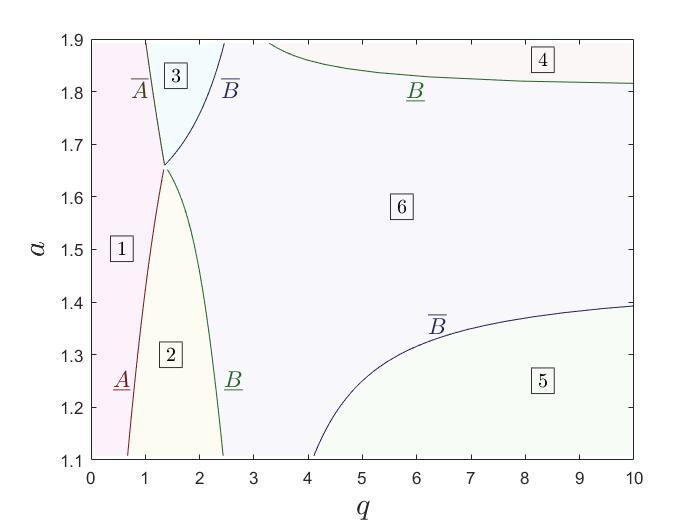}
	\caption{The Feasible Region is Divided into Six Conditions.}
	\label{fig:split}
\end{figure}
Specifically, we plot
Figure \ref{fig:split} with $\mu_1=1,\mu_2=1.5, a + b = 3$ and $c = 0.7$, when ranging $a$ and $q$ over $[1.1, 1.9]$ and $[0,10]$ respectively. 
Figure 1 shows that the feasible region can be divided into six disjoint subsets, where borderlines are in the form of $\bar{A}$, $\underline{A}$, $\bar{B}$, $\underline{B}$, $\bar{C}$ and $\underline{C}$ that are functions of input $(\theta,q)$ defined in the proof of this lemma, see equation (B.5) in Appendix B.1.

For each condition, the following theorem presents a closed-form optimal value of the bivariate moment problem (\ref{eq:bivariatePrimal}).
\begin{theorem} \label{thm: vp}
	Given a non-empty ambiguity set $\F(\theta)$ and $q>0$, the optimal value $v_P(q;\theta)$ can be characterized as
	{\small
	$$v_P(q;\theta)=\left\{ 
	\begin{aligned}
	& \mu_1 + \mu_2 - q\cdot \frac{a+b - 2c}{ab - c^2}\ & &\text{if Condition 1 holds};\\
 &\frac{b-1}{2b}\left( (q+Q_b) - \frac{b-c}{b-1}\mu_1 \right) + \mu_1+\mu_2 - q\ & &\text{if Condition 2 holds};\\
& \frac{a-1}{2a}\left( (q+Q_a) - \frac{a-c}{2a}\mu_2 \right) + \mu_1+\mu_2 - q\ & &\text{if Condition 3 holds};\\ 
& \frac{b-1}{2b} \left(\frac{b-c}{b-1}\mu_1 - (q-Q_b) \right)\ & &\text{if Condition 4 holds};\\ 
&  \frac{a-1}{2a}\left(\frac{a-c}{a-1}\mu_2 -(q-Q_a) \right) \ & &\text{if Condition 5 holds};\\
	& \frac{1}{2} ( Q_c- q+\mu_1+\mu_2 )\ & &\text{if Condition 6 holds}.
	\end{aligned}
	\right.
	$$}
\end{theorem}

We relegate the proof of this theorem to Appendix B.
In this theorem, we characterize the optimal value $v_P(q;\theta)$ of problem (\ref{eq:bivariatePrimal}) by six different cases.
To obtain the optimal distributions, we extend the primal-dual approach in \cite{guo2022unified}, designed for the univariate moment problems, to our two-dimensional problems, and successfully identify a corresponding optimal distribution for each case stated in Lemma B.2.1-B.2.6.

There are several interesting observations of the optimal values in Theorem \ref{thm: vp}. First of all, the optimal value under condition 6 is the same as that in a univariate pooling problem. Specifically, let $\bar{X}=X_1+X_2$ and we consider the univariate problem $\sup_{\mathbb{P}\in \bar{\mathcal{F}}_1} \mathbb{E}_{\mathbb{P}} [(\bar{X} -q)_+]$ with the
 ambiguity set $$\bar{\mathcal{F}}_1=\left\{\mathbb{P}\in \mathbb{M}(\mathbb{R}_+) : \;
\mathbb{E}_{\mathbb{P}} [\bar{X}] = \mu_1+\mu_2,  \;\;
\mathbb{E}_{\mathbb{P}}  [\bar{X}^2] = a\mu_1^2+b\mu_2^2+2c\mu_1\mu_2 
\right\}.$$
We can verify that the optimal value under condition 6 is the same as that in this univariate pooling problem via (\ref{eq: scarfOptQLg}) in Lemma \ref{scarflemma}.

Secondly, the optimal value under condition 2 and that under condition 3 are symmetric through replacing $b$ by $a$ and $\mu_1$ by $\mu_2$. 
This symmetric property also applies to the optimal values under condition 4 and condition 5, while the optimal values under condition 1 and condition 6 are self-symmetric. 

With respect to the optimal distributions, from Lemma B.2.2 and Lemma B.2.4 in Appendix B, we can see that the optimal distributions under condition 2 and condition 4 share the same formulation, though their optimal values are different. Specifically, both optimal distributions are in the form of
 $$ \begin{cases}
	x^{(1)}=(q-Q_b,0)  & \text{w.p. } \, p_1=\frac{b-1}{2b}+\frac{q(b-1)+\mu_1(c-b)}{2bQ_b};\\
	x^{(2)}=(q+Q_b,0) & \text{w.p. }  \, p_2=\frac{b-1}{2b}-\frac{q(b-1)+\mu_1(c-b)}{2bQ_b};\\
	x^{(3)}=(c\mu_1, b\mu_2) & \text{w.p. } \, p_3=\frac{1}{b}.\\
	\end{cases}$$
To illustrate, we plot the optimal distribution under condition 2 in Figure \ref{fig:worstcase1}. 
\begin{figure}[htbp]
    \centering
   \subfigure[Condition 2]{
\includegraphics[width=0.4\columnwidth]{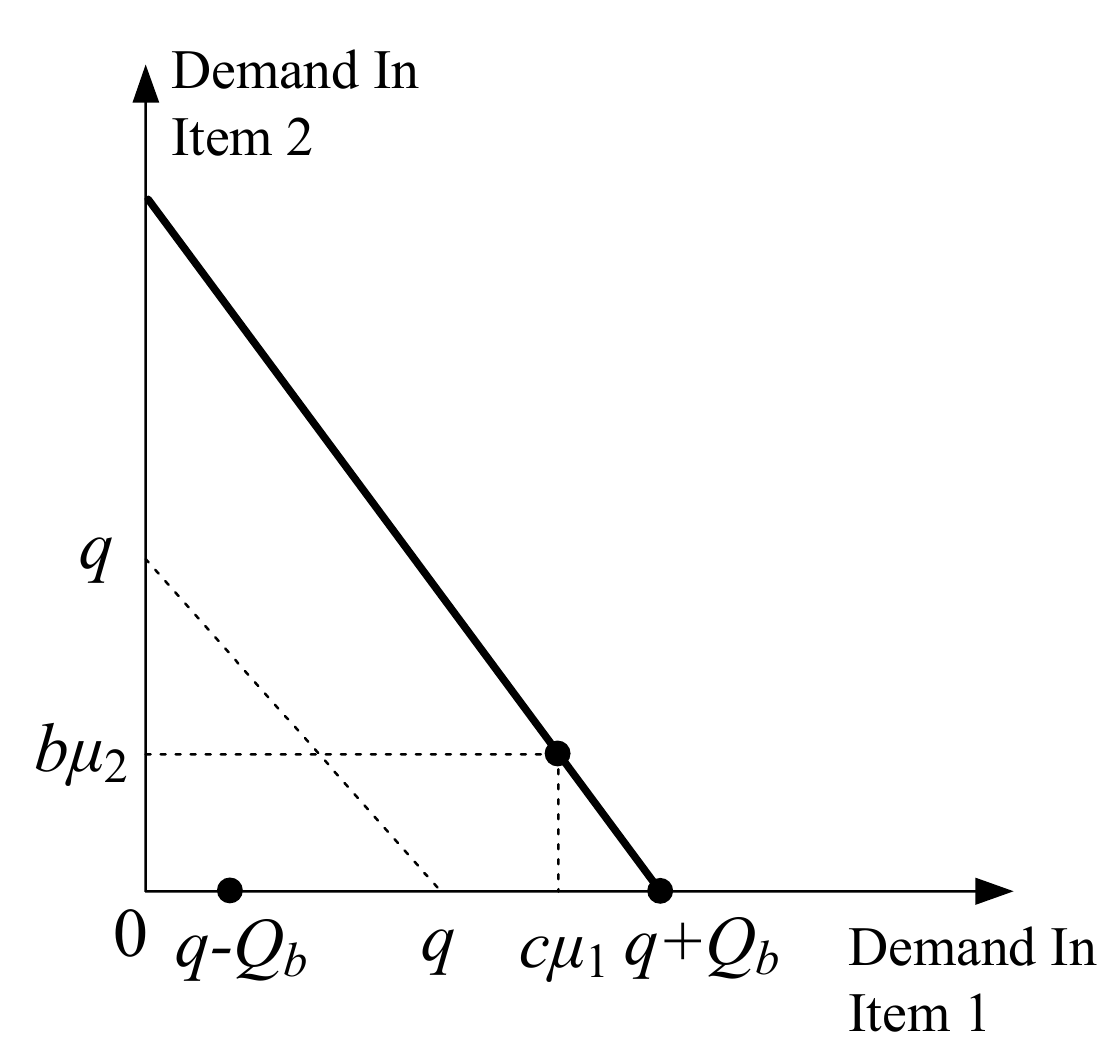}
\label{fig:worstcase1}
   } 
    \subfigure[Condition 4]{ \includegraphics[width=0.4\columnwidth]{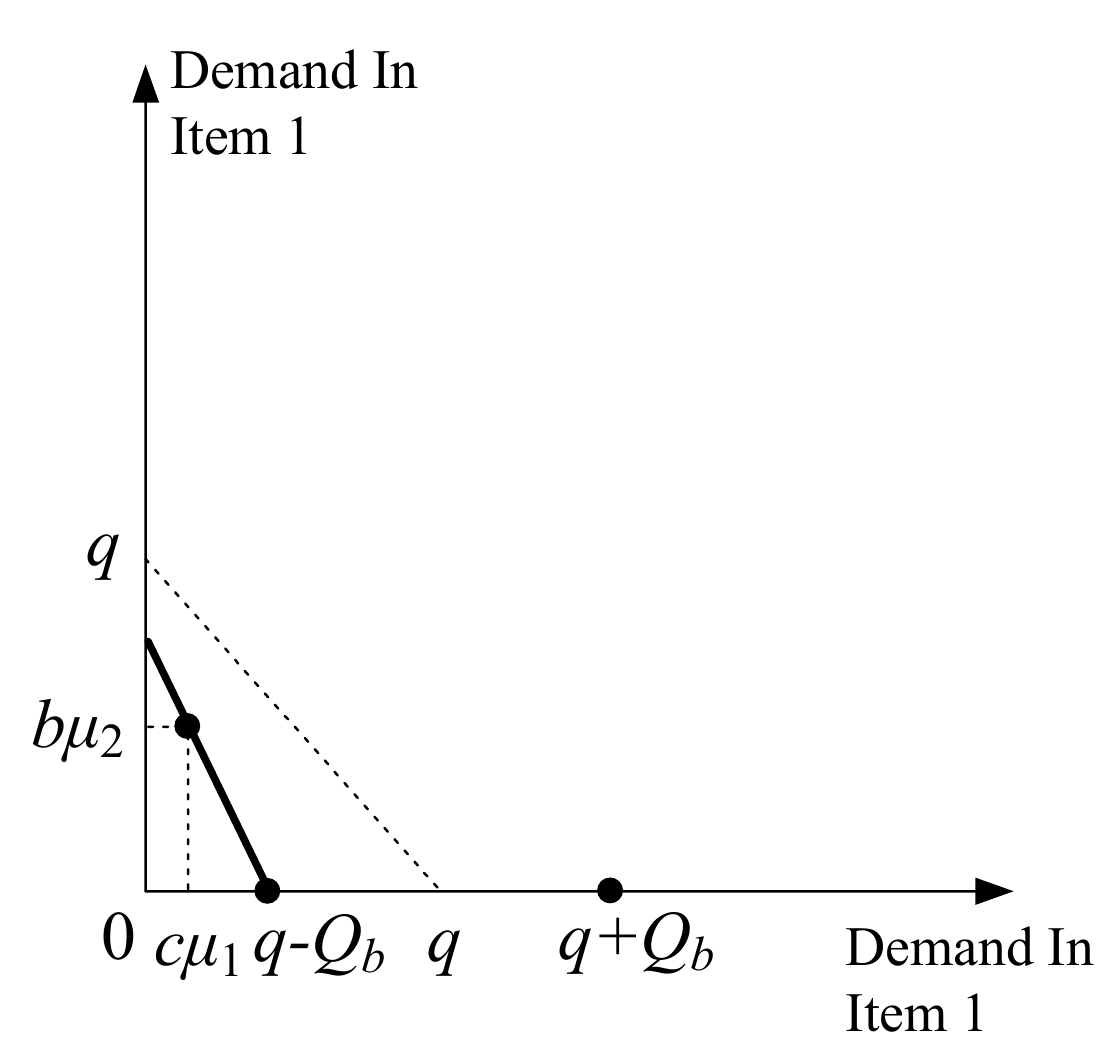}
    \label{fig:worstcase2}}
    \caption{The Three Support Points under Conditions 2 and 4}
\end{figure}
The points on the solid line together with $(q-Q_b,0)$ are potential points in the optimal support, as they satisfy the complementary slackness conditions. In addition, condition 2 indicates that
$c\mu_1+b\mu_2 -q\geq Q_b >0$,
so the objective value at point $x^{(3)}$ is strictly positive with $\left(x_1^{(3)}+x_2^{(3)}-q \right)_+=c\mu_1+b\mu_2-q>0$. 
We also plot the optimal distribution under condition 4 in Figure \ref{fig:worstcase2}. This condition indicates that
$c\mu_1+b\mu_2 -q\leq -Q_b <0$,
which implies a zero objective value at $x^{(3)}$. In short, $x^{(3)}$ affects the objective value under condition 2 but not under condition 4. 

In fact, the optimal distributions in other conditions can also be interpreted similarly.
Note that the optimal distributions under conditions 3 and 5 also share the same formulation, as they are symmetric to those under condition 2 and condition 4 respectively. Furthermore, notice that the optimal distribution under condition 2 becomes invalid under condition 1, because $x^{(1)}_1=q-Q_b<0$, which contradicts the non-negativity of random variables. When $q$ shrinks from above $Q_b$ to below $Q_b$, condition 2 is switched to condition 1, which is consistent with the illustration in Figure \ref{fig:split} with $q$ crossing $\underline{A}$. Consequently,  the optimal distribution under condition 1 always takes $(0,0)$ in its optimal support from Lemma A.1. Lastly, as the optimal value under condition 6 is the same as that in the univariate pooling problem, it is not surprising to see each support point $x$ of optimal distributions lay on either $x_1+x_2=q-Q_c$ or $x_1+x_2=q+Q_c$ from Lemma B.2.6, which is consistent with the supporting points of the univariate pooling problem via (\ref{eq: scarfOptQLg}) in Lemma \ref{scarflemma}.

\subsection{Extension of the Loss Function}
Beyond the two-piecewise linear objectives, recent literature also considers the multi-piece quadratic loss functions in newsvendor models \cite{ghosh2022non,ghosh2021new}. Specifically, quadratic losses are used to measure the cost severity of critical perishable commodities \cite{ghosh2022non}. Moreover, such multi-piece quadratic functions are also employed in the robust portfolio selection problem
to represent the lower partial moment as a measure of risk \cite{chen2011tight}.



In this subsection, we study the bivariate moment problem under half-space support with a multi-piece quadratic objective function.
As it is computationally challenging to obtain such a solution in closed form, we propose a numerical approach to solve this problem. Specifically, we provide a semi-definite programming (SDP) formulation by applying the sum-of-squares technique \cite{powers1998algorithm} to our dual problem. We relegate the proof to Appendix C.

\begin{proposition}
Given a random vector $\boldsymbol{X} \in \mathbb{M} (\mathbb{R}^2)$ with its ambiguity set $\F(\theta)$ defined in (\ref{eq: amSet}) and a matrix of coefficients $W \in \mathbb{R}^{6 \times K}$ for a $K$ piece-wise quadratic objective, we denote each piece as follows:
$$\ell_{k}(\boldsymbol{X})=w_{1k} + w_{2k} X_1 + w_{3k}X_2 +w_{4k} X_1^2 + w_{5k} X_2^2 + w_{6k} X_1X_2$$
for $k \in \{1,...,K\}$. The following bivariate moment problem 
$$\sup_{\mathbb{P}\in \mathcal{F}(\theta)}\ \mathbb{E}_{\mathbb{P}} \left[\max_{k=1, ..., K} \left\{ 
 \ell_{k}(\boldsymbol{X}) \right\} \right]$$
can be solved by an SDP:
\begin{equation*}
\begin{split}
\inf_{\boldsymbol{z},{G},{H}}\quad &z_1 + \mu_1 z_2 + \mu_2 z_3 + \Sigma_{11} z_4 +  \Sigma_{22} z_5 + \Sigma_{12} z_6  \\
\mathrm{s.t.}\quad &M \left(\boldsymbol{z} - \boldsymbol{w}^{(k)}, \boldsymbol{g}^{(k)}, \boldsymbol{h}^{(k)} \right) \succeq 0,\qquad k=1, ..., K
\end{split}
\end{equation*}
with $\boldsymbol{z} \in  \mathbb{R}^{6}, {G} \in \mathbb{R}^{3 \times K}, {H} \in \mathbb{R}^{3 \times K}$ and the matrix  $M(\tilde{\boldsymbol{z}},\boldsymbol{g},{\boldsymbol{h}})$ defined as follows:
\begin{equation*}
\begin{split}
	 \begin{pmatrix}
	\tilde{z}_4 & 0 & -{g}_1 & 0 & -{h}_1 & -{g}_2 \\
	0 & \tilde{z}_6 +2{g}_1 & 0 & {h}_1 & -{h}_2 & -h_3 \\
	-{g}_1 & 0 & \tilde{z}_5 & {h}_2 & 0 & -{g}_3 \\
	0 & {h}_1 & {h}_2 & \tilde{z}_2 + 2{g}_2 & {h}_3 & 0 \\
	-{h}_1 & -{h}_2 & 0 & {h}_3 & \tilde{z}_3 + 2{g}_3 & 0 \\
	-{g}_2 & -{h}_3 & -{g}_3 & 0 & 0 & \tilde{z}_1 
	\end{pmatrix}
\end{split}
\end{equation*}
where $\boldsymbol{w}^{(k)}, \boldsymbol{g}^{(k)}, \boldsymbol{h}^{(k)}$ are the $k$th column vectors of matrices $W$, ${G}$, ${H}$ respectively. 

\label{propositionSOS}
\end{proposition}

As far as we know, it is mathematically challenging to obtain an exact reformulation for piece-wise polynomial objectives with either a higher degree or more than two variables. This is because the sum-of-squares technique may no longer apply in these settings.
Specifically, each dual constraint of our problem is to ensure that the corresponding polynomial is non-negative. 
Indeed, our analysis uses the result by Hilbert in 1888, who
showed in \cite{hilbert1888darstellung} that every non-negative polynomial with $n$ variables and degree $2d$ could be represented as sums of squares of other polynomials if and only if either $n = 1$ or $2d = 2$ or $n = 2$ and $2d = 4$. Note that the formulation of our proposition is a bivariate polynomial ($n=2$) with fourth-order
($2d=4$). Therefore, the sum of squares technique can be applied. 
Similar challenges for a higher degree or more than two variables are also presented by Bertsimas and Popescu, who briefly review the computational tractability in moment problems \cite{bertsimas2005optimal}.

\section{The Bivariate Newsvendor Problem} \label{section: Impact of Pooling}

 In the previous section, we provided a closed-form solution for the inner bivariate moment problem. In this section, we study the outer DRO problem. Specifically,
 we consider the DRO newsvendor problem as follows:
 \begin{equation}
 \label{eq:bivariateNewsvendor}
\inf_{q\geq 0} \left\{\sup_{\mathbb{P}\in \mathcal{F}(\theta)}\ \mathbb{E}_{\mathbb{P}} \left[\left(X_1+X_2 -q \right)_+\right] + (1-\eta)q \right\},
 \end{equation} 
where $\mathcal{F}(\theta)$ is the mean-covariance ambiguity set in (\ref{eq: amSet}), and the critical ratio $\eta$ is a given constant in $(0,1)$. 

In the following, we first propose an approach to solve this problem in Section \ref{sec:framwork}, which also reveals the relation between the optimal order $q^*$ and the moment parameter $\theta$ in an analytical form. In Section \ref{subsection: The Effects of Inventory Pooling}, we analyze the benefit of incorporating the covariance information into the ambiguity set, as well as the impact on the total inventory. In Section \ref{sec: bpm vs upm}, we show that it is important to keep the covariance information instead of compressing such information into one dimension.


\subsection{Optimal Order Quantity}\label{sec:framwork}
We introduce the following procedure in the box below to solve problem (\ref{eq:bivariateNewsvendor}). 
In terms of the computational techniques to solve $q^*_i$ in (\ref{eq: localQ}), 
we first note that each $A_i$ essentially corresponds to an interval of $q$, which can be solved from a quadratic equation. As a result, we obtain each $A_i$ in a closed form shown in Table B.2 in Appendix B.1.
Furthermore, we note that
$v_p(q;\theta) + (1-\eta)q$ is the sum of a linear term and a square root of a quadratic term of $q$, which is strictly convex in $q$. Therefore, it is not difficult to obtain the unique stationary points of (\ref{eq: localQ}) shown in Appendix E.
Since an optimal solution locates at either a stationary point or a boundary point of the interval $A_i$, 
(\ref{eq: localQ}) can be solved explicitly. To summarize, this framework is capable of solving the optimal order $q^*$ systematically.

\begin{shadedbox}
  1. Identify local solutions separately: for each condition $i$ of the six conditions in Theorem \ref{thm: vp}, solve 
		\begin{equation} \label{eq: localQ}
		q_i^* = \mathop{\arg\min}_{q\in A_i} \left\{v_p(q;\theta) + (1-\eta)q \right\}   
		\end{equation}
		where $A_i=\{q \in \R_+: q \text{ satisfies condition } i\}$.
  
  \noindent
  2. Obtain global solutions jointly: solve
  $$
		q^* = \mathop{\arg\min}_{q\in \{q_1^*, ..., q_6^* \}} \left\{v_p(q;\theta) + (1-\eta)q\right\}.
  $$
\end{shadedbox}

\subsection{The Effects of Inventory Pooling}
\label{subsection: The Effects of Inventory Pooling}
Inventory pooling is frequently employed as an operational strategy to mitigate demand uncertainty: combining inventory allows the company to decrease demand variability, cut operational costs, and boost profits, especially if the component market demands are negatively correlated \cite{swinney2012inventory}. The pooling strategy often results in a centralized inventory system \cite{bimpikis2016inventory,eppen1979note,su2008bounded}. We illustrate the difference between centralized and decentralized structures of inventory systems through Figure \ref{fig:twostructures}. In Figure \ref{centralized}, we show a single supplier serving two demand streams, which we call the centralized system; in Figure \ref{decentralized}, we show two suppliers deciding order quantities separately, which we call the decentralized system.
\begin{figure}[htbp]
    \centering
   \subfigure[Centralized System]{
    \begin{minipage}[t]{0.35\linewidth}
        \centering
\includegraphics{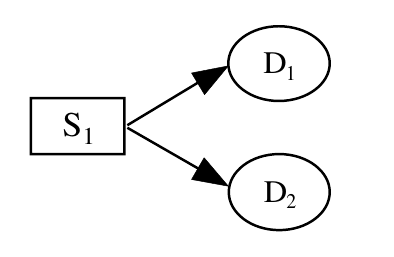}
\label{centralized}
    \end{minipage}
   }   
    \subfigure[Decentralized System]{ \begin{minipage}[t]{0.35\linewidth}
        \centering
\includegraphics{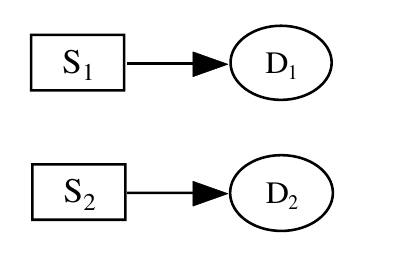}
    \label{decentralized}
    \end{minipage}}
    \caption{Two Structures of Inventory Systems}
    \label{fig:twostructures}
\end{figure}


From the perspective of pooling, we consider our model in (\ref{eq:bivariateNewsvendor}) as a \text{\it bivariate centralized model} (BCM). In the following, we compare the centralized model with the decentralized model. We first denote individual ambiguity sets of demands $X_1$ and $X_2$ by 
$$\mathcal{F}_1(\theta)=\left\{\mathbb{P}\in \mathbb{M}(\mathbb{R}_+) : \mathbb{E}_{\mathbb{P}} [X_1] = \mu_1, \mathbb{E}_{\mathbb{P}} [X_1^2] = \Sigma_{11} \right\},$$
$$\mathcal{F}_2(\theta)=\left\{\mathbb{P}\in \mathbb{M}(\mathbb{R}_+) : \mathbb{E}_{\mathbb{P}} [X_2] = \mu_2, \mathbb{E}_{\mathbb{P}} [X_2^2] = \Sigma_{22} \right\},$$
as the marginal ambiguity sets of $\mathcal{F}(\theta)$ in (\ref{eq:bivariateNewsvendor}). Note that these marginal ambiguity sets neglect the covariance information in the original ambiguity set $\mathcal{F}(\theta)$.
Now we are ready to propose the \text{\it bivariate decentralized model} (BDM) as follows: 
\begin{equation} 
\label{eq: BDM}
\inf_{q_1,q_2\geq 0} \left\{v_1(q_1;\theta) + v_2(q_2;\theta)+ (1-\eta)(q_1+q_2) \right\},   
\end{equation}
where $v_1(q_1;\theta)$ and $v_2(q_2;\theta)$ are the optimal values of 
$$\sup_{\mathbb{P}\in \mathcal{F}_1(\theta)} \mathbb{E}_{\mathbb{P}} [(X_1 -q_1)_+] \;\; \text{ and }
\sup_{\mathbb{P}\in \mathcal{F}_2(\theta)} \mathbb{E}_{\mathbb{P}} [(X_2 -q_2)_+] \;, $$ respectively. Note that we can decompose the BDM to be the sum of two marginal univariate DRO problems, i.e., 
\begin{equation*} 
\inf_{q_1 \geq 0} \left\{v_1(q_1;\theta) + (1-\eta)q_1 \right\} + 
\inf_{q_2 \geq 0} \left\{v_2(q_2;\theta) + (1-\eta)q_2 \right\}.   
\end{equation*}
Therefore, the optimal solution of the BDM can be obtained 
through applying Scarf's result \cite{scarf1958min} on each individual univariate problem with 
\begin{equation*} 
q_i^*=\left\{
 \begin{array}{ll}
        \mu_i+\frac{\sqrt{\Sigma_{ii}-\mu_i^2}}{2}\frac{2\eta-1}{\sqrt{\eta(1-\eta)}}, & \text{if } \frac{ \Sigma_{ii}-\mu_i^2}{\Sigma_{ii}}<\eta <1;\\
        0, & \text{if } 0 \leq \eta \leq \frac{ \Sigma_{ii}-\mu_i^2}{\Sigma_{ii}} ,
        \end{array} \right.   
\end{equation*}
for $i=1,2$.

In the rest of this subsection, we compare the optimal values and optimal solutions between the BCM and the BDM. By doing so, we demonstrate the role of the correlation coefficient $\rho$ in the BCM over various critical ratios $\eta$. With regard to the optimal values, the BCM always outperforms the BDM. Intuitively, the BDM is inclined to adopt an overly conservative decision, since this model neglects the correlation between demands. As an illustration, we plot the relative gap of optimal values between the BCM and the BDM over correlation coefficient $\rho$ and critical ratio $\eta$ in Figure \ref{fig:comparison1_O}, when fixing the other moment information with $\mu_1=1, \mu_2=1, a=2, b=6$. Specifically, this relative gap is defined as the ratio:
$$\kappa(\rho,\eta)=\frac{V_{\text{BDM}}(\theta)-V_{\text{BCM}}(\theta)}{V_{\text{BCM}}(\theta)},$$
where $V_{\text{BCM}}$ and $V_{\text{BDM}}$ are optimal values of the BCM and the BDM, respectively.
 \begin{figure}[h]
	\centering
	\includegraphics[scale=0.5]{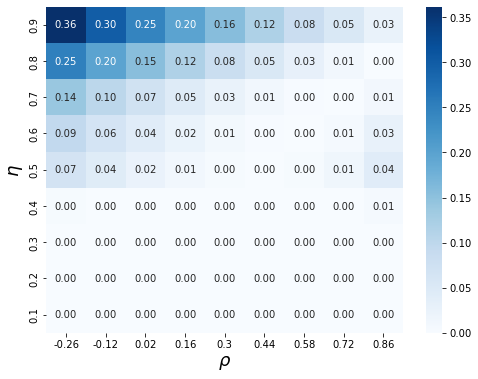}
	\caption{Relative Gap between BCM and BDM}
	\label{fig:comparison1_O}
\end{figure}
In Figure \ref{fig:comparison1_O}, it is not surprising to see $\kappa(\rho,\eta)$ is relatively large when $\rho$ is small. This phenomenon is consistent with the well-known 
\text{\it pooling effect}, i.e.,
inventory pooling can significantly reduce operational costs when demands are negatively correlated. Moreover, when $\eta$ is small, both models tend to accept a zero inventory so that $\kappa(\rho,\eta)$ is zero as expected.

An unexpected observation is that the BCM also enjoys an advantage when $\rho$ is close to 1 and $\eta$ is in the intermediate range. This is due to another effect  
which we call
the \text{\it shifting effect}. Before introducing the shifting effect, we first illustrate this observation with a concrete example through further fixing $\eta=0.5$. In such a case, we obtain $V_{\text{BDM}}=2$. For $V_{\text{BCM}}(\rho)$, by the method introduced in Section \ref{sec:framwork}, we have
 \begin{equation*}
     \begin{split}
         V_{\text{BCM}}(\rho) = \frac{1}{5} \sqrt{5-5\rho^2} + \frac{\sqrt{5}}{10} \rho + \frac{3}{2}.
     \end{split}
 \end{equation*}
 As shown in Figure \ref{fig:eta=0.5},
  \begin{figure}[htb]
	\centering
	\includegraphics[scale=0.45]{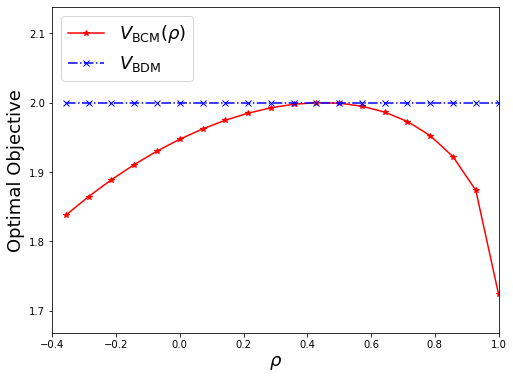}
	\caption{$V_{\text{BDM}}$ and $V_{\text{BCM}}(\rho)$ when $\eta = 0.5$.}
	\label{fig:eta=0.5}
\end{figure}
$V_{\text{BCM}}(\rho)$ is a concave function on $[-\sqrt{5}/5, 1]$\footnote{If $\rho<-\sqrt{5}/5$, then Assumption \ref{pro:1} will be violated.}  with  the optimal value $V_{\text{BCM}}(\rho^*)=2$ obtained at $\rho^*=\sqrt{5}/5$. The increase of $V_{\text{BCM}}(\rho)$ over $\rho\in [-\sqrt{5}/5, \sqrt{5}/5 ]$, as expected, is because of the pooling effect. On the other hand, the shifting effect dominates the pooling effect over $\rho\in (\sqrt{5}/5, 1]$, resulting in
a decrease in the objective function of $V_{\text{BCM}}(\rho)$. To explain the shifting effect, we first introduce the following proposition with $\rho=1$, whose proof is relegated to Appendix D.
\begin{proposition}\label{p3}
 Suppose that two nonnegative random variables $X_1$ and $X_2$ satisfy
 $$\mathbb{E} [X_1] = \mu_1, \mathbb{E}[X_2] = \mu_2, \mathbb{E}[X_1^2] = a\mu_1^2, \mathbb{E} [X_2^2]=b\mu_2^2,$$
 and assume $1 \leq a \leq b$ without loss of generality. 
 If $X_1$ and $X_2$ are perfectly correlated with $\rho=1$, then it must be $X_1 \geq \left(1-\sqrt{\frac{a-1}{b-1}}\right) \mu_1 \geq 0$.
\end{proposition}
From this proposition, we can see that a perfect correlation with $\rho=1$ still benefits the BCM by shifting the restriction of uncertain demand $X_1$ from $X_1 \geq 0$ to $X_1 \geq \left(1-\sqrt{(a-1)/(b-1)}\right) \mu_1 $, which is considered as the shifting effect. Consequently, this effect tightens the ambiguity set and thus results in a less conservative decision.
\begin{figure}[htbp]
    \centering
   \subfigure[The Value of $\max_{F\in \mathcal{F}(\theta)}\mathbb{P}(X_1\leq \xi)$]{
    \begin{minipage}[t]{0.45\linewidth}
        \centering
\includegraphics[scale=0.45]{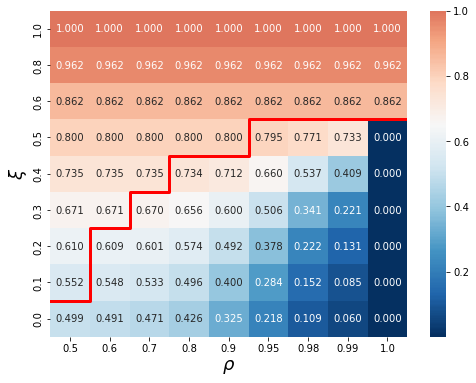}
\label{fig:PrX1LeqXi}
    \end{minipage}
   }   
    \subfigure[The Value of $\max_{F\in \mathcal{F}_1(\theta)}\mathbb{P}(X_1\leq \xi)$ ]{ \begin{minipage}[t]{0.45\linewidth}
        \centering
\includegraphics[scale=0.45]{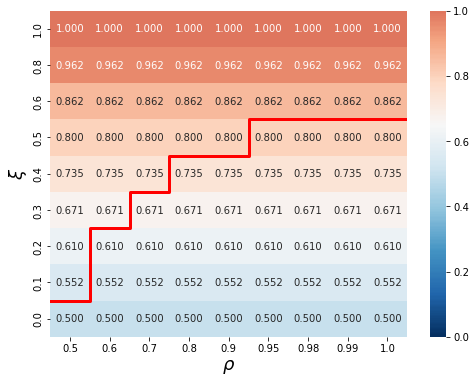}
    \label{fig:PrX1LeqXi1}
    \end{minipage}}
 \caption{Shifting Effect}
\label{fig: shifting effect}
\end{figure}

To further illustrate the shifting effect, we note that a large $\rho$ benefits the BCM by reducing the probability of occurrences of a small demand realization, thus tightening the ambiguity set and making the decision less conservative. 
As an illustration, we plot the upper bound of probability $\mathbb{P}(X_1\leq \xi)$ over correlation $\rho$ in Figure \ref{fig: shifting effect}, when fixing the same parameters with $\mu_1=1, \mu_2=1, a=2, b=6$. On the upper left of the red dividing line, the values of $\max_{F\in\mathcal{F}(\theta)}\mathbb{P}(X_1\leq \xi)$ remain the same as $\max_{F\in\mathcal{F}_1(\theta)}\mathbb{P}(X_1\leq \xi)$. On the lower right of this red line, we observe that $$\max_{F\in\mathcal{F}(\theta)}\mathbb{P}(X_1\leq \xi)<\max_{F\in\mathcal{F}_1(\theta)}\mathbb{P}(X_1\leq \xi),$$
and the shifting effect takes place. 
Besides, for  $\xi\leq 0.5$, the shifting effect keeps getting stronger as the value of $\rho$ increases. In the extreme case when $\rho=1$ and $\xi\leq 0.5<1-\sqrt{0.2}$,
Figure \ref{fig:PrX1LeqXi} shows that $\max_{F\in\mathcal{F}(\theta)}\mathbb{P}(X_1\leq \xi)=0$, which verifies Proposition \ref{p3}.

Another counter-intuitive observation is that a centralized system does not necessarily reduce the optimal total order quantity, 
\begin{figure}[htbp]
	\centering
	\includegraphics[scale=0.55]{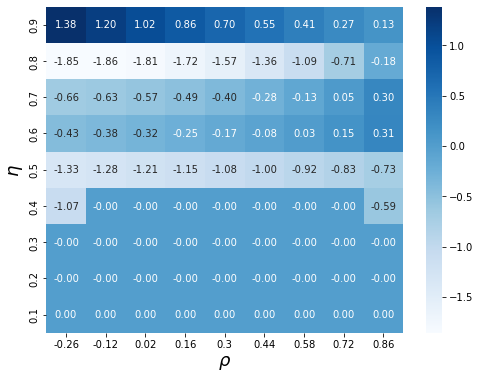}
	\caption{Gap of Optimal Orders: $q^*_{\text{BDM}}-q^*_{\text{BCM}}$ }
	\label{fig:comparision1_q}
\end{figure}
i.e., the optimal order quantity $q^*_{\text{BCM}}$ may be larger than $q^*_{\text{BDM}}$ with $q^*_{\text{BDM}}=q^*_{1 \;\text{BDM}}+q^*_{2 \;\text{BDM}}$. 
We illustrate this observation in Figure \ref{fig:comparision1_q} with the same set of moment parameters as in Figure \ref{fig: shifting effect}.
In Figure \ref{fig:comparision1_q}, a very large $\eta$ (e.g., $\eta\geq0.9$) results in the worst-case demand distribution taking condition 6 in Theorem \ref{thm: vp}. In such a case, the conventional pooling effects dominate \cite{bimpikis2016inventory, eppen1979note}. For an instance with $\eta=0.9$, $\rho=0.3$, the centralized optimal order $q^*_{\text{BCM}}=5.61$, whereas the decentralized orders are $q^*_{1 \;\text{BDM}}=2.33$ and $q^*_{2 \;\text{BDM}}=3.98$ respectively and pooling indeed reduces the optimal total order quantity. However, for a medium-large $\eta$, the individual optimal order quality could become 0 under sufficiently large demand uncertainty. For an instance with $\eta=0.7, \rho=0.3$, the centralized optimal order $q^*_{\text{BCM}}=1.84$, whereas the decentralized orders are $q^*_{1 \;\text{BDM}}=1.44$ and $q^*_{2 \;\text{BDM}}=0$ respectively. In some sense, demand pooling can reduce demand variability and smooth the change of the optimal orders along with different critical ratios $\eta$, and during this process, the optimal total order quantity under the centralized system may be larger than that under the decentralized system.

\subsection{The Effects of Ambiguity Pooling} \label{sec: bpm vs upm}

For bivariate moment problems, a straightforward technique of relaxation, often considered in the literature, is to model the problem as a univariate problem \cite{tian2008moment}. Such relaxation
can simplify the problem definition, reduce the solution complexity, and ease the analysis. 
More precisely,
this technique of relaxation results in a {\it univariate centralized model} (UCM) as follows:
\begin{equation} \label{eq:UCM}
\inf_{q\geq 0} \left\{\sup_{\mathbb{P}\in \bar{\mathcal{F}}(\theta)} \mathbb{E}_{\mathbb{P}} [(\bar{X}-q)_+] + (1-\eta)q \right\} \end{equation}
where we treat $\bar{X}=X_1+X_2$
with $\bar{\mu}=\mu_1+\mu_2$, $\bar{\Sigma}=\Sigma_{11}+2\Sigma_{12}+\Sigma_{22}$, and define
$$\bar{\mathcal{F}}(\theta) = \{\mathbb{P}\in \mathbb{M}(\mathbb{R}_+) : \mathbb{E}_{\mathbb{P}} [\bar{X}] = \bar{\mu}, \mathbb{E}_{\mathbb{P}} \left[\bar{X}^2 \right] =\bar{\Sigma} \}.$$
Unlike the BDM which neglects the covariance, the UCM compresses the uncertainty of mean and covariance into information in the one-dimensional space. We call such a maneuver 
{\it ambiguity pooling}.

In the rest of this subsection, we demonstrate the effects of ambiguity pooling by comparing the optimal value and ambiguity sets of the BCM and the UCM. In Figure \ref{fig:comparison3_O}, we plot the relative gap between the optimal values through fixing the same part of moment parameters with $\mu_1=1, \mu_2=1, a=2, b=6$.
\begin{figure}[htbp]
	\centering
	\includegraphics[scale=0.5]{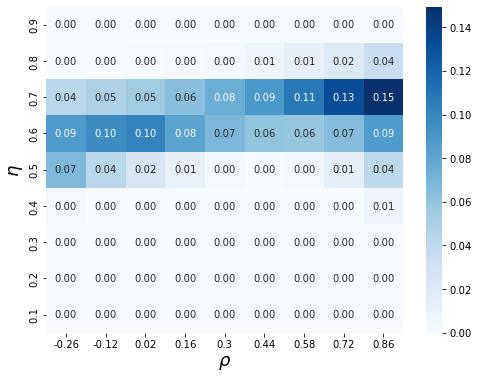}
	\caption{Relative Gap between BCM and UCM}
	\label{fig:comparison3_O}
\end{figure}
As shown in Figure \ref{fig:comparison3_O}, the BCM enjoys a $5\% - 15\%$ performance improvement with a medium to large $\eta$. It is not surprising to see that the relative gap is zero when $\eta$ is small or very large: when $\eta$ is small, both models will adopt a zero inventory, thus the gap is zero; when $\eta$ is very large, condition 6 in Theorem \ref{thm: vp} holds, resulting in the same optimal values for both models.

An essential observation is that the ambiguity pooling process may lead to a loss of information on the uncertainty
under half-space support. To be precise, we consider the following ambiguity set
\begin{equation*} 
\hat{\F}(\theta) = \left\{\mathbb{P}\in \mathbb{M}(\mathbb{R}^2) :  \;
\mathbb{E}_{\mathbb{P}} [X_1+X_2] = \bar{\mu},  \;\;  
\mathbb{E}_{\mathbb{P}} \left[(X_1+X_2)^2 \right] = \bar{\Sigma},\;\;
X_1+X_2 \geq 0
\right\}, 
\end{equation*}	
and the UCM can be reformulated as follows:
 \begin{equation*}
\inf_{q\geq 0} \left\{\sup_{\mathbb{P}\in \hat{\mathcal{F}}(\theta)}\ \mathbb{E}_{\mathbb{P}} \left[\left(X_1+X_2 -q \right)_+\right] + (1-\eta)q \right\}.
 \end{equation*} 
 It is obvious that $\F(\theta) \subset\hat{\F}(\theta)$, which means the ambiguity pooling process potentially results in a more conservative decision. Next,
 we shall provide an example to demonstrate that the optimal distribution of problem $\sup_{\mathbb{P}\in \hat{\F}} \mathbb{E}_{\mathbb{P}} [(X_1+X_2-q)_+]$ does not belong to $\F$; in other words, $\F(\theta)$ is a proper subset of $ 
 \hat{\F}(\theta) $. We take $\mu_1=1, \mu_2=1, a=2, b=6,c=1$ as an example, in which $\bar{\mu}=2$ and $\bar{\Sigma}=10$. Based on the Scarf bound in Lemma 1, a class of optimal distributions for problem $\sup_{\mathbb{P}\in \hat{\F}} \mathbb{E}_{\mathbb{P}} [(X_1+X_2-q)_+]$ with total $K+1$ mass points can be described as follows:
\begin{align*}
\hat{\mathbb{P}}^*=\begin{cases}
(0,0),\,& \text{w.p. } p_1=1-\frac{\bar{\mu}^2}{\bar{\Sigma}}=0.6\\
(x_1^{(k)},x_2^{(k)}),& \text{w.p. } p_2^{(k)},\; \text{ for all }\, k=1,...,K
\end{cases}
\end{align*}
with $\sum_{k=1}^K p_2^{(k)}=\frac{\bar{\mu}^2}{\bar{\Sigma}}=0.4$ and $x_1^{(k)}+x_2^{(k)}=\frac{\bar{\Sigma}}{\bar{\mu}}=5$ for all $k=1,...,K$.
To show $\hat{\mathbb{P}}^* \notin \mathcal{F}$ by contradiction, we first assume $\hat{\mathbb{P}}^*\in \mathcal{F}$. Thus, the following constraints must hold for $X_1$:
\begin{align*}
\mathbb{E}[X_1]=\sum_{k=1}^Kx_1^{(k)} p_2^{(k)}=\mu_1=1,\, \text{ and } \,
\mathbb{E}[X_1^2]=10-10+2=2.
\end{align*}
 As a result, the second moment of $X_2$ must satisfy
\begin{equation*}   
\mathbb{E}[X_2^2]=
\sum_{k=1}^K\left(\frac{\bar{\Sigma}}{\bar{\mu}}-x_1^{(k)}\right)^2p_2^{(k)}
=\frac{\bar{\Sigma}^2}{\bar{\mu}^2} \sum_{k=1}^K p_2^{(k)} -2\frac{\bar{\Sigma}}{\bar{\mu}} \mathbb{E}[X_1] + \mathbb{E}[X_1^2]
=\left(1-\frac{2\mu_1}{\bar{\mu}}\right)\bar{\Sigma}+\Sigma_{11}=2.
\end{equation*}
 This contradicts with $\mathbb{E}[X_2^2]=\Sigma_{22}=6$. Therefore, $\hat{\mathbb{P}}^* \notin \mathcal{F}$.


Lastly, we remark that the assumption on the half-space support plays an important role in the above result. 
If the half-space support is replaced by the full-space support for both the BCM and the UCM, then the two models will result in the same optimal values by the projection theorem in \cite{popescu2007robust}.

\section{Conclusion and Further Discussion}\label{sec:conclusion}
To summarize, we study a bivariate distributionally robust optimization problem with the mean-covariance ambiguity set and half-space support. For a class of widely-adopted objective functions in inventory management, option pricing, and portfolio selection, we obtain closed-
form tight bounds and optimal distributions of the inner problem. Furthermore, we show that under the distributionally robust setting,
a centralized inventory system does not always reduce the optimal total inventory. This
contradicts with the belief that a centralized inventory system can reduce the total
inventory. In addition, we identify two effects, a conventional pooling effect and a novel shifting
effect. Their combination determines the benefit of incorporating the covariance information
in the ambiguity set. Finally, we demonstrate the importance of keeping the covariance information in the ambiguity set instead of compressing the information
into one dimension through numerical experiments.

It is worth mentioning that our result of the bivariate moment problem is also useful to derive a closed-form upper bound of the multivariate moment problem $\sup_{\mathbb{P}\in \mathcal{F}}\ \mathbb{E}_{\mathbb{P}} \left[\ell\left( \boldsymbol{w}^T \boldsymbol{X} \right)\right]$ in (\ref{eq:innerWorstCaseExpectationProblem}) with $\ell (x) = \max\left\{u_1x +v_1, u_2x +v_2 \right\}$. Such multivariate problem is known to be computationally challenging, and Natarajan et al. provide a mathematically tractable upper bound for the expected piecewise linear utility function in a numerical way \cite{natarajan2010tractable}. In consideration of our results, we treat the $n$-dimensional random variable $\bX$ as consisting of a sequence of low-dimensional random variables $\bX_i$ with $\bX=(\bX_1,\bX_2,\ldots,\bX_I)$ and $\bX_i \in \R^{{n_i}}$. In turn, we can decompose the ambiguity set $\mathcal{F}$ into corresponding marginal ambiguity sets $\mathcal{F}_i$. Based on such a decomposition, we obtain an upper bound of the multivariate problem as follows:
$$
\sup_{\mathbb{P}\in \mathcal{F}}\ \mathbb{E}_{\mathbb{P}} \left[\ell\left( \boldsymbol{w}^T \boldsymbol{X} \right)\right]
\leq \sum_{i\in I}\sup_{\mathbb{P}_i\in \mathcal{F}_i} \mathbb{E}_{\mathbb{P}_i} \left[\max\{u_1\boldsymbol{w}_i^T{\boldsymbol{X}_i}+v_{1i}, u_2\boldsymbol{w}_i^T{\boldsymbol{X}_i}+v_{2i}\}\right]
$$
for all possible $v_{1i}$ and $v_{2i}$ satisfying $\sum_{i=1}^{I}v_{1i}=v_1$ and $\sum_{i=1}^{I}v_{2i}=v_2$. Therefore, if $n_i\leq 2$ for all $i$, we can provide a closed-form upper bound. In terms of numerical algorithms, we can incorporate $\sum_{i=1}^{I}v_{1i}=v_1$ and $\sum_{i=1}^{I}v_{2i}=v_2$ as two constraints to obtain a better upper bound. 
Note that the estimation of the $n \times n$ covariance matrix is usually unreliable in practice if $n$ is large enough. Our approach relies on only a number of marginal well-estimated $2 \times 2$ covariance submatrices, which is considered to be robust against data perturbations. On the other hand, the weakness of our approach is that we only incorporate $\lfloor n/2 \rfloor$ submatrices at most. The open question is how to make more use of the covariance information so as to provide a better bound for this multivariate moment problem.

Another possible research direction is to incorporate the shape of distributions including symmetry, unimodality, or convexity into the ambiguity set in order to provide less conservative decisions appropriate to their respective scenarios, which we also leave for further study.




\bibliographystyle{apalike} 
 \bibliography{cas-refs}






\newpage

\section*{Appendix A: Proof of Proposition 1}

    Rewrite the loss function
	\begin{align*}
	\ell(x) &= \max\{u_1x+v_1, u_2x+v_2 \} \\
	&= \max\{0, (u_2-u_1)x+v_2-v_1\} + (u_1x+v_1) \\
	&= \big((u_2-u_1)x+v_2 - v_1 \big)_+ + (u_1x+v_1) \\
	&= (u_2-u_1) \left(x+ \frac{v_2-v_1}{u_2-u_1} \right)_+ + (u_1x+v_1).
	\end{align*}
	Let $\tilde{\boldsymbol{X}} = \boldsymbol{w}\circ \boldsymbol{X}$, then $\mathbb{E}_{\mathbb{P}} [\tilde{\boldsymbol{X}}] = \boldsymbol{w}\circ \boldsymbol{\mu}$ and $\mathbb{E}_\mathbb{P} \big[\tilde{\boldsymbol{X}}\tilde{\boldsymbol{X}}^T \big] = \boldsymbol{w}\boldsymbol{w}^T\circ \Sigma$. Thus, the ambiguity set $\mathcal{F}$ of random vector $\boldsymbol{X}$ is equivalent to the ambiguity set $\tilde{\mathcal{F}}$ of random vector $\tilde{\boldsymbol{X}}$.
	Therefore,
	\begin{equation*}
	    \begin{aligned}
	\sup_{\mathbb{P}\in \mathcal{F}} \mathbb{E}_{\mathbb{P}} \left[\ell\left( \boldsymbol{w}^T \boldsymbol{X} \right)\right] 
	= &\sup_{\mathbb{P}\in \mathcal{F}} \mathbb{E}_{\mathbb{P}} \left[(u_2-u_1) \left(\boldsymbol{w}^T \boldsymbol{X} + \frac{v_2-v_1}{u_2-u_1} \right)_+ + \left(u_1 \boldsymbol{w}^T \boldsymbol{X} +v_1\right) \right] \\
 =\ & (u_2-u_1) \sup_{\mathbb{P}\in \mathcal{F}} \mathbb{E}_{\mathbb{P}}\left[\left(\boldsymbol{w}^T \boldsymbol{X} + \frac{v_2-v_1}{u_2-u_1} \right)_+ \right] + u_1 \boldsymbol{w}^T \boldsymbol{\mu} +v_1  \\
	=\ &(u_2-u_1) \sup_{\mathbb{P}\in \tilde{\mathcal{F}}} \mathbb{E}_{\mathbb{P}} \left[\left(\boldsymbol{1}^T \tilde{\boldsymbol{X}} + \frac{v_2-v_1}{u_2-u_1} \right)_+ \right] + u_1 \boldsymbol{w}^T\boldsymbol{\mu} +v_1.
	    \end{aligned}
	\end{equation*}

\section*{Appendix B: Proof of Theorem 1}
As proved in \cite{bertsimas2002relation}, problem $\sup_{\mathbb{P}\in \mathcal{F}} \mathbb{E}_{\mathbb{P}} [\ell( \boldsymbol{w}^T \boldsymbol{X} )]$ is NP-hard for
general cases. In this paper, we study the bivariate case ($n=
2$). Specifically, we consider the following bivariate moment problem
\setcounter{equation}{0}
\renewcommand{\theequation}{B.\arabic{equation}}
 \begin{equation}
 \label{eq:bivariatePrimal_app}
v_P(\theta;q)=\sup_{\mathbb{P}\in \mathcal{F}(\theta)}\ \mathbb{E}_{\mathbb{P}} \left[\left(X_1+X_2 -q \right)_+\right],
\end{equation}
where $\theta=(\mu_1,\mu_2,\Sigma_{11},\Sigma_{22},\Sigma_{12})$ and the ambiguity set $\F(\theta)$ is described as follows
\begin{equation} \label{eq: amSet_app}
\F(\theta)=\{\mathbb{P}\in \mathbb{M}(\mathbb{R}_+^2) : \mathbb{E}_{\mathbb{P}} [X_1] = \mu_1, \mathbb{E}_{\mathbb{P}} [X_2] = \mu_2, \mathbb{E}_{\mathbb{P}}\left[X_1^2 \right] = \Sigma_{11}, \mathbb{E}_{\mathbb{P}}\left[X_2^2 \right] = \Sigma_{22}, \mathbb{E}_{\mathbb{P}}\left[X_1X_2 \right] = \Sigma_{12} \}.
\end{equation}
 For simplicity, we denote $(\Sigma_{11},\Sigma_{22},\Sigma_{12}):=(a\mu_1^2, b\mu_2^2, c\mu_1\mu_2)$ for the rest of the Appendix. Thus, the covariance matrix $M$ and correlation $\rho$ can be represented by  
\begin{equation} \label{eq: M_apx}
M=\left[
\begin{array}{cc}
(a-1)\mu_1^2&(c-1)\mu_1\mu_2\\
(c-1)\mu_1\mu_2&(b-1)\mu_2^2
\end{array}
\right]
\;\; \text{ and } \;\;
\rho= \frac{c-1}{\sqrt{(a-1)(b-1)}}.
\end{equation} 
Note that parameters $a$, $b$ and $c$ follow Assumption 1, which characterizes the feasibility of the primal problem (\ref{eq:bivariatePrimal_app}).
To exclude the trivial cases in which the ambiguity set only contains the distribution with a fixed single-point marginal distribution,  without loss of generality,  we assume that $a>1$ and $b>1$ in the remaining discussion.

The dual of problem (\ref{eq:bivariatePrimal_app}) is given as
\begin{equation}\label{eq:dual}
    \begin{aligned}
        &&\quad \quad  v_D(\theta;q)=\inf_{\boldsymbol{z}} \quad&  z_1+z_2\mu_1+z_3\mu_2+z_4a\mu_1^2+z_5b\mu_2^2+z_6c\mu_1\mu_2 \\
	&&  {s.t. }\quad h_1(x;\boldsymbol{z}):=& z_1+z_2x_1+z_3x_2+z_4x_1^2+z_5x_2^2+z_6x_1x_2\geq 0,  \quad \forall (x_1, x_2)\in \R_+^2\\
	&&  \quad h_2(x;\boldsymbol{z}):= q+z_1&+(z_2-1)x_1+(z_3-1)x_2+z_4x_1^2+z_5x_2^2+z_6x_1x_2\geq 0, \quad\forall (x_1, x_2)\in \R_+^2,
    \end{aligned}
\end{equation}
with the dual variable $\boldsymbol{z}=(z_1,z_2,z_3,z_4,z_5,z_6) \in \R^6$.

\setcounter{pro}{1}


Define six conditions of $(\theta, q)$ in Table \ref{conditiontable},
 \setcounter{table}{0}
\renewcommand{\thetable}{B.\arabic{table}}
\begin{table}[htbp]
\centering
        \begin{tabular}{|c|c|c|c|c|c|}
            \hline
             Condition 1 & Condition  2 & Condition  3 & Condition  4 & Condition 5 & Condition  6 \\
            \hline
              $Q_a \geq q,$  & $Q_b < q,$ & $Q_a < q,$  & $Q_b < q,$ & $Q_a < q,$ & $Q_a > |\zeta_a|,$ \\
              $Q_b \geq q$ & $Q_b \leq \zeta_b$ & $Q_a \leq \zeta_a$ & $Q_b \leq -\zeta_b$ & $Q_a \leq -\zeta_a$ & $Q_b > |\zeta_b|$ \\
             \hline
        \end{tabular}
        \caption{Six Conditions}
             \label{conditiontable}
        \end{table}
 
 \noindent where $\zeta_a=a\mu_1+c\mu_2-q$ and $\zeta_b=c\mu_1+b\mu_2-q$, 
 and
\begin{equation*}
\begin{aligned}
Q_a &= \sqrt{q^2 - 2q \frac{a-c}{a-1}\mu_2 + \frac{ab-c^2}{a-1}\mu_2^2}, \\
Q_b &= \sqrt{q^2 - 2q \frac{b-c}{b-1}\mu_1 + \frac{ab-c^2}{b-1}\mu_1^2}, \\
Q_c &= \sqrt{q^2 - 2q(\mu_1 + \mu_2) + a\mu_1^2 + b\mu_2^2 + 2c\mu_1\mu_2}.
\end{aligned}
\end{equation*}


	For simplicity of the later proof, we denote 
 \begin{equation}\label{AB}
     \overline{A}=\frac{ab-c^2}{2(a-c)}\mu_2, \;\; \underline{A}=\frac{ab-c^2}{2(b-c)}\mu_1,\;\;
	\overline{B}=\overline{A}+\frac{\overline{D} \cdot(\overline{C}-\underline{C}) }{2(a-c) \overline{C}}, \;\; \underline{B}=\underline{A}+\frac{\underline{D} \cdot (\underline{C}-\overline{C})}{2(b-c) \underline{C}},
 \end{equation}
	where
 \begin{equation}
	\overline{C}=(a-1)\mu_1+(c-1)\mu_2,\;\; \underline{C}=(c-1) \mu_{1}+(b-1) \mu_{2},
 \end{equation}
 and
 \begin{equation}\label{fz3}
	\overline{D}=(a-1) (a \mu_1 + c \mu_2)>0, \;\;\underline{D}=(b-1)(c\mu_1+b\mu_2)>0.
 \end{equation}
 Note that $\overline{A},\overline{B},\overline{C},\overline{D}$ and $\underline{A},\underline{B},\underline{C},\underline{D}$ are symmetric if we replace $\mu_1$ by $\mu_2$ and $a$ by $b$.

    We also present four terms that will be frequently employed for the later analysis:
	\begin{align}
	&a\mu_1+c\mu_2-\overline{B}=\frac{a \overline{C}^2+(a-1)(b-1)-(c-1)^2}{(a-1)\overline{C}}\mu_2^2,    \label{eq: diff1}\\
	&c\mu_1+b\mu_2-\underline{B}=\frac{b \underline{C}^2+(a-1)(b-1)-(c-1)^2}{(b-1)\underline{C}}\mu_1^2,    \label{eq: diff2}\\
	&\overline{A}-\underline{A}=\frac{ab-c^2}{(a-c)(b-c)}(\underline{C}-\overline{C}), \label{eq: diffA}\\
	&\overline{B}-\underline{B}=\frac{(\overline{C}-\underline{C})}{\overline{C}\underline{C}}E, \label{eq: diffB}
	\end{align}
	where $E=(a-1)c \mu_1^2+\big(ab-a-b+c^2\big)\mu_1\mu_2+(b-1)c \mu_2^2$. Because of the condition $(a-1)(b-1)-(c-1)^2\geq0$ in Assumption 1, it is easy to prove that
$E \geq (a-1)c \mu_1^2+2(c-1)c\mu_1\mu_2+(b-1)c \mu_2^2=c\V ar(X_1+X_2) >0$ .

 In Appendix B.1, we prove parameters $(\theta,q)$ must satisfy one and precisely one of six conditions in Table \ref{conditiontable}. Then we provide the worst-case distributions for every condition and prove the optimality in Appendix B.2. Through the analyses of the worst-case distribution for each condition, Theorem 1 can be proved.

\subsection*{\textbf{Appendix B.1.} Proof of Lemma 2}

To show that every feasible input $(\theta,q)$ satisfies one and exactly one of six conditions, we first show that conditions in Table \ref{conditiontable} are equivalent to those in  (\ref{cond1}-\ref{cond6}). By simple algebra, we can verify the following equivalency with
	$$    Q_a^2 \leq \zeta_a^2 \Leftrightarrow \overline{C} (q -\overline{B}) \leq 0, \;\;
 Q_b^2 \leq \zeta_b^2 \Leftrightarrow \underline{C} (q - \underline{B}) \leq 0,
	$$
$$Q_a\geq q \Leftrightarrow 2q(a-c)\leq(ab-c^2)\mu_2, \;\; Q_b\geq q \Leftrightarrow 2q(b-c)\leq(ab-c^2)\mu_1.$$
Thus, we can rewrite conditions in Table \ref{conditiontable} as follows: 
 \begin{align}
     & \textbf{condition 1: }
     2q(a-c)\leq(ab-c^2)\mu_2,\;\; 2q(b-c)\leq(ab-c^2)\mu_1;
     \label{cond1}\\
     & \textbf{condition 2: }
     2q(b-c)>(ab-c^2)\mu_1, \;\;\underline{C} (q - \underline{B}) \leq 0, \;\; \zeta_b \geq 0;
     \label{cond2}\\
     & \textbf{condition 3: }
     2q(a-c)>(ab-c^2)\mu_2, \;\;\overline{C} (q - \overline{B}) \leq 0, \;\; \zeta_a \geq 0; 
     \label{cond3}\\
     & \textbf{condition 4: }
    2q(b-c)>(ab-c^2)\mu_1, \;\;\underline{C} (q - \underline{B}) \leq 0, \;\; \zeta_b \leq 0; 
    \label{cond4}\\
    & \textbf{condition 5: }
    2q(a-c)>(ab-c^2)\mu_2, \;\;\overline{C} (q - \overline{B}) \leq 0, \;\; \zeta_a \leq 0; 
    \label{cond5}\\
    & \textbf{condition 6: }
    \underline{C} (q - \underline{B}) > 0, \;\; \overline{C} (q - \overline{B}) > 0.
    \label{cond6}
 \end{align}

Furthermore, according to the conditions in (\ref{cond1}-\ref{cond6}), we will discuss the feasible region of $q$ under each condition based on the input $\theta$ in different cases. The results are summarized in Table \ref{tab:tableRegionQ}. 
	\begin{table}[htbp]
		\renewcommand{\arraystretch}{1.3}
     \small
		\begin{tabular}{c|c|c|c|c|c|c}
			\hline
			\multirow{2}*{\diagbox{Cond}{Cases}}&\multirow{2}*{$a>c>b$}&\multirow{2}*{$b>c>a$}&\multicolumn{4}{c}{$a>c,b>c$}\\
			\cline{4-7}
			~&~&~&
   $\overline{C}\geq \underline{C},\underline{C}>0$
   &$\overline{C}\geq \underline{C},\underline{C}<0$
   &$\overline{C}\leq \underline{C},\overline{C}>0$
   &$\overline{C}\leq \underline{C},\overline{C}<0$\\
			\hline
			\text{condition 1}&$0\leq q\leq \overline{A}$&$0\leq q\leq \underline{A}$&$0\leq q\leq \overline{A}$&$0\leq q\leq \overline{A}$&$0\leq q\leq \underline{A}$&$0\leq q\leq \underline{A}$\\
			\text{condition 2}& -&$\underline{A}<q\leq\underline{B}$&-&-&$\underline{A}<q\leq\underline{B}$&$\underline{A}<q\leq\underline{B}$\\
			\text{condition 3}&$\overline{A}<q\leq\overline{B}$&-&$\overline{A}<q\leq\overline{B}$&$\overline{A}< q\leq\overline{B}$&-&-\\
			\text{condition 4}&-&-&-&$q\geq\underline{B}$&-&-\\
			\text{condition 5}&-&-&-&-&-&$q\geq\overline{B}$\\
			\text{condition 6}&$q>\overline{B}$&$q>\underline{B}$&$q>\overline{B}$&$\overline{B}<q<\underline{B}$&$q>\underline{B}$&$\underline{B}<q<\overline{B}$\\
			\hline
		\end{tabular}
		\caption{Feasible Regions of $q$ in Different Cases}
		\label{tab:tableRegionQ}
	\end{table}
Through this table, we can clearly see that the intersection of feasible regions of $q$ in each column is empty, and the union of those is $\R_+$. In addition, as the six cases of input $\theta$ are also mutually exclusive, every feasible input $(\theta,q)$ satisfies one and exactly one of six conditions.
Now, we are ready to see how to obtain feasible regions of $q$ in each case.

	\textbf{Case $1$} ($a>c>b$): In this case, we have $a>c>b>1$ based on $b>1$ in Assumption 1, and we can first verify that
 \begin{equation}\label{fz1}
 \overline{C}>0, \;\;\underline{C}>0,\;\;
\text{and} \;\;\overline{C}-\underline{C}=(a-c)\mu_1+(c-b)\mu_2>0.
\end{equation}
 
Condition 1 in \eqref{cond1} can be rewritten as
 $$q\leq \frac{ab-c^2}{2(a-c)}\mu_2=\overline{A}\;\; \text{and}\;\;q\geq\frac{ab-c^2}{2(b-c)}\mu_1=\underline{A}.$$
According to H$\ddot{o}$lder's inequality, we have $(\mathbb{E}[X_1X_2])^2\leq \mathbb{E}[X^2_1]\mathbb{E}[X^2_2]$, i.e.,
\begin{equation}\label{fz2}
    c^2\leq ab.
\end{equation}
Thus, $\overline{A}\geq 0$ and $\underline{A}\leq 0$  due to \eqref{fz2} and $b<c$. In all,
condition 1 can be represented as $0\leq q\leq\overline{A}$, as $q \in \R_+$.
 
Condition 3 in \eqref{cond3} can be rewritten as
 $$q>\overline{A},\;\; q\leq\overline{B},\;\; \text{and}\;\; q \leq a\mu_1+c\mu_2,$$ 
because of $\overline{C}>0$ in \eqref{fz1} and $a>c$.
 According to (\ref{eq: diff1}), we have 
\begin{equation}\label{fz4}
a\mu_1+c\mu_2-\overline{B}=\frac{a \overline{C}^2+(a-1)(b-1)-(c-1)^2}{(a-1)\overline{C}}\mu_2^2>0,
\end{equation}
due to  $\overline{C}>0$ in \eqref{fz1} and $(a-1)(b-1)\geq (c-1)^2$ in Assumption 1. In addition, we can verify that $\overline{B}-\overline{A}=\frac{\overline{D} \cdot(\overline{C}-\underline{C}) }{2(a-c) \overline{C}} > 0$ due to \eqref{fz1} and $\overline{D}>0$ in \eqref{fz3}. In all, condition 3 can be simplified as $\overline{A}<q\leq\overline{B}$. 

Condition 6 in \eqref{cond6} can be written as $$q>\overline{B}\;\; \text{and}\;\; q> \underline{B},$$
due to inequality \eqref{fz1}. Moreover, we have $\overline{B} > \underline{B}$, as $\overline{B}-\underline{B}>0$ by equation \eqref{eq: diffB} and inequality \eqref{fz1}. In all, condition 6 can be represented as $q>\max\{\overline{B},\; \underline{B}\}=\overline{B}$.

Similar to the previous analysis,  the feasible region of $q$ under condition 2 in  (\ref{cond2}) and that under condition 4 in (\ref{cond4}) are empty with $q<\underline{A}<0$. Due to $a\mu_1+c\mu_2-\overline{B}>0$ in \eqref{fz4}, we know that $q\leq\overline{B}$ and $q\geq a\mu_1+c\mu_2$ together are incompatible, which means that the feasible region of $q$ under condition 5 in \eqref{cond5} is also empty.

In summary, the space $q\geq0$ is fulfilled by condition 1, condition 3, and condition 6 in this case, and all feasible regions are mutually exclusive, which are summarized in the first column of Table \ref{tab:tableRegionQ}.  
	
\textbf{Case $2$} ($b>c>a$): 
 Note that the condition on the input $\theta$ under case 2 ($b>c>a$) is symmetric to that under case 1 ($a>c>b$) in the sense of swapping $a$ and $b$. In fact, condition 2 and condition 3 are symmetric; condition 4 and condition 5 are symmetric; condition 1 and condition 6 are self-symmetric. Therefore, by the same analysis in case 1, we summarize the result in the second column of Table \ref{tab:tableRegionQ}.

\textbf{Case $3$} ($a>c,b>c, \overline{C} \geq \underline{C}, \underline{C}> 0$):  
In this case, we can first verify that 
\begin{equation}\label{case3_1}
\underline{A}\geq0,\;\;\overline{A}\geq0,\;\;\overline{A}-\underline{A}=\frac{ab-c^2}{(a-c)(b-c)}(\underline{C}-\overline{C})\leq0.
\end{equation}
due to (\ref{fz2}). 

Condition 1 in (\ref{cond1}) can be rewritten as 
$$q\leq\overline{A}\;\;\text{and}\;\;q\leq\underline{A},$$
due to $a>c,b>c$. 
As $\overline{A} \leq \underline{A}$  holds by (\ref{case3_1}) and $q \in \R_+$, condition 1 can be represented as $0\leq q\leq\overline{A}$.

 Condition 3 in \eqref{cond3} can be written as $$q>\overline{A},\;\; q\leq\overline{B},\;\; \text{and} \;\; q\leq a\mu_1+c\mu_2,$$
 because of $\overline{C}>0$ and $a>c$. According to (\ref{eq: diff1}), we have $a\mu_1+c\mu_2-\overline{B}>0$ due to $\overline{C}\geq\underline{C}>0$ in this case and 
$(a-1)(b-1)\geq (c-1)^2$ in Assumption 1. In addition, we can verify that $\overline{B}-\overline{A}=\frac{\overline{D} \cdot(\overline{C}-\underline{C}) }{2(a-c) \overline{C}} \geq 0$ due to $\overline{D}>0$ in \eqref{fz3} and $\overline{C}\geq\underline{C}$ in this case. Thus, $a\mu_1+c\mu_2> \overline{B}\geq \overline{A}$. In all, the condition can be simplified as $\overline{A}<q\leq\overline{B}$.

Condition 6 in (\ref{cond6}) can be written as 
$$q>\overline{B}\;\; \text{and}\;\; q>\underline{B},$$
due to $\overline{C}\geq\underline{C}>0$. Moreover, we have $\overline{B}\geq \underline{B}$, as $\overline{B}- \underline{B}\geq0$ by equation (\ref{eq: diffB}) and $\overline{C}\geq\underline{C}$. In all, condition 6 
can be simplified as $q>\overline{B}$.

Similar to the previous analysis, we can verify that 
$\underline{B}-\underline{A}=\frac{\underline{D} \cdot (\underline{C}-\overline{C})}{2(b-c) \underline{C}}\leq0$.
Thus we know that $q\leq\underline{B}$ and $q>\underline{A}$
together are incompatible, which means that the feasible region of $q$ under condition 2 in \eqref{cond2} and condition 4 in \eqref{cond4} are empty. In addition, due to 
$a\mu_1+c\mu_2-\overline{B}>0$, we know that $q\geq a\mu_1+c\mu_2$ and $q\leq\overline{B}$ together are incompatible, which means that the feasible region of $q$ under condition 5 in \eqref{cond5} is also empty. 

In summary, the space $q\geq0$ is fulfilled by condition 1, condition 3, and condition 6 in this case, and all feasible regions are mutually exclusive, which are summarized in the third column of Table \ref{tab:tableRegionQ}.

\textbf{Case $4$} ($a>c,b>c, \overline{C} \geq \underline{C}, \underline{C}< 0$): 
In this case, we first prove that $\overline{C} > 0$ by contradiction. Assume $\overline{C}\leq 0$ and $\underline{C}<0$, i.e., $(a-1)\mu_1 \leq (1-c)\mu_2$ and $(b-1)\mu_2 < (1-c)\mu_1$. If we cancel $\mu_1$ and $\mu_2$, it leads to $(a-1)(b-1)<(1-c)^2$, which contradicts the feasibility requirement in Assumption 1. 

Furthermore, we can verify that 
\begin{equation*}
\underline{A}\geq0,\;\;\overline{A}\geq0,\;\;\overline{A}-\underline{A}\leq0.
\end{equation*}
By the same analysis in case 3, the feasible regions of $q$ under condition 1 and condition 3 are the same as those in case 3.



Condition 4 in \eqref{cond4} can be written as 
$$q>\underline{A},\;\;q\geq\underline{B}\;\;\text{and}\;\;q\geq c\mu_1+b\mu_2,$$
because of $b>c$ and $\underline{C}< 0$.
According to (\ref{eq: diff2}), we have that 
$c\mu_1+b\mu_2-\underline{B}<0$ due to $\underline{C}<0$ in this case and 
$(a-1)(b-1)-(c-1)^2\geq0$ in Assumption 1. In addition, we verify that $\underline{B}-\underline{A}=\frac{\underline{D} \cdot (\underline{C}-\overline{C})}{2(b-c) \underline{C}}\geq0$ due to $\underline{D}>0$ in  \eqref{fz3} and $\underline{C}<0$ together with 
$\underline{C}\leq\overline{C}$ in this case. In all, condition 4 can be simplified as $q\geq\underline{B}$.

Condition 6 in \eqref{cond6} can be written as 
$$q<\underline{B}\;\;\text{and}\;\;q>\overline{B},$$
because of $\overline{C} > 0$ and $\underline{C}<0$.
Moreover, we have $\overline{B}\leq \underline{B}$, as $\overline{B}- \underline{B}\leq0$ by equation (\ref{eq: diffB}) and $\overline{C}\geq\underline{C}$. In all, condition 6 
can be simplified as $\overline{B}<q<\underline{B}$.

Similar to the previous analysis, due to 
$c\mu_1+b\mu_2-\underline{B}<0$, we know that 
$q\geq\underline{B}$ and $q\leq c\mu_1+b\mu_2$ together are incompatible, which means that 
the feasible region of $q$ under condition 2 in (\ref{cond2}) is empty. In addition, due to $a\mu_1+c\mu_2-\overline{B}>0$, we know that  
$q\leq\overline{B}$ and $q\geq a\mu_1+c\mu_2$ together are incompatible, which means that the feasible region of $q$ under condition 5 in (\ref{cond5}) is empty.

In summary, the space $q\geq0$ is fulfilled by condition 1, condition 3, condition 4, and condition 6 in this case, and all feasible regions are mutually exclusive, which are summarized in the fourth column of Table \ref{tab:tableRegionQ}.

	\textbf{Case $5$} ($a>c,b>c, \overline{C} < \underline{C}, \overline{C}> 0$):
 Note that the condition about the input $\theta$ under case 5 ($a>c,b>c, \overline{C} < \underline{C}, \overline{C}> 0$) is symmetric to that under case 3 ($a>c,b>c, \overline{C} \geq \underline{C}, \underline{C}> 0$). As the symmetry 
 of condition 2 and condition 3  and the self-symmetry of condition 1 and condition 6, by the same analysis in case 3, 
 we summarize the result in the fifth column of Table \ref{tab:tableRegionQ}.

\textbf{Case $6$} ($a>c,b>c, \overline{C} < \underline{C}, \overline{C}<0$):  
 Note that the condition about the input $\theta$ under case 6 ($a>c,b>c, \overline{C} < \underline{C}, \overline{C}<0$) is symmetric to that under case 4 ($a>c,b>c, \overline{C} \geq \underline{C}, \underline{C}< 0$).
 Note that condition 2 and condition 3 are symmetric; condition 4 and condition 5 are symmetric; condition 1 and condition 6 are self-symmetric. Therefore, by the same analysis in case 4, we summarize the result in the last column of Table \ref{tab:tableRegionQ}.

\subsection*{\textbf{Appendix B.2.} Proof of Theorem 1}
In this subsection, we provide the optimal values and optimal distributions under each condition in Lemma \ref{lem:1}-\ref{lem:6}, and Theorem 1 follows immediately from these lemmas. 

Before diving into details, we first present our intuitions on how to derive the optimal distribution under each condition. Let $\bz^*$ be the dual optimal solution, so the dual feasibility implies $h_1(\bx;\bz^*) \geq 0$ and $ h_2(\bx;\bz^*) \geq 0$ for all $\bx \in \R^2_+$ in \eqref{eq:dual}. Moreover, the complementary slackness conditions indicate that $h_1(\bx^*;\bz^*)=0$ or $h_2(\bx^*;\bz^*)=0$ for every $\bx^*$ in optimal support. Instead of focusing on $h_1$ and $h_2$ directly,
geometrically, we consider a curved surface $S$ with
$$ S = \left\{(x_1, x_2, t)\in  \mathbb{R}^3 : z_1^*+x_1z_2^*+x_2z_3^*+x_1^2z_4^*+x_2^2z_5^*+x_1x_2z_6^* = t  \right\},$$
and a folded hyperplane $T$ with 
$$T = \left\{(x_1, x_2, t)\in \mathbb{R}_+^3 : \max\{x_1+x_2-q, 0\} = t  \right\}.$$
Accordingly, $h_1, h_2 \geq 0$ is equivalent to $S$ above $T$, and $h_1=0$ or $h_2=0$ is equivalent to $S$ touching $T$. Through this interpretation, we can quickly analyze the characteristics of optimal support. Based on these characteristics, we successfully identify an optimal distribution under each condition. In this derivation, although we are motivated by the framework for solving univariate moment problems introduced in \cite{guo2022unified}, it also requires a lot of trial and error to obtain an optimal distribution of our bivariate moment problem. As an illustration, we plot $S$ and $T$ with parameters $\mu_1=\mu_2=1, a=1.1, b=1.015, \rho=0.7, q=1$ in Figure \ref{3Dfigure}.
In this figure, the points of the solid line together with $(0,q-Q_a)$ are touching points of $S$ and $T$. In fact, the curved surface $S$ is a paraboloid. 
\setcounter{figure}{0}
\renewcommand{\thefigure}{B.\arabic{figure}}
\begin{figure}[h]
	\centering
	\includegraphics[width=0.65\linewidth]{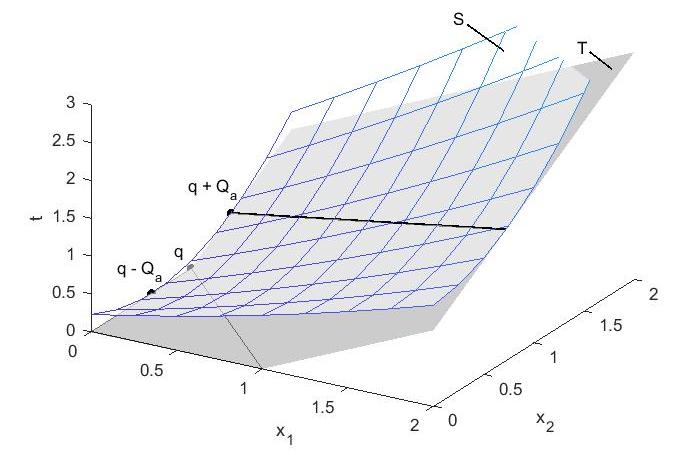}
  \caption{An illustration of $S$ and $T$}
	\label{3Dfigure}
\end{figure}

Now, we are ready to present six lemmas corresponding to each condition. In each proof, we first verify the primal feasibility. Then, we provide a dual solution and show its feasibility. Lastly, we verify the  zero duality gap and thus the optimality.

\setcounter{lemma}{0}
\renewcommand{\thelemma}{B.2.\arabic{lemma}}

\begin{lemma}
	\label{lem:1}
	Suppose that $ Q_a \geq q, Q_b \geq q$, and Assumption 1 holds. The optimal distribution for problem \ref{eq:bivariatePrimal_app}
	 can be characterized as follows
	
	(i) if $b>c$, 
	\begin{equation}
	\label{distr:zeroline1}
	\begin{cases}(0,0), & \text{w.p. } \, p_{1}=\frac{(a-1)(b-1)-(c-1)^2}{(a b-c^{2})} \\ (\frac{a b-c^{2}}{b-c} \mu_{1}, 0),  & \text{w.p. } \, p_{2}=\frac{(b-c)^{2}}{\left(a b-c^{2}\right) b} \\ (c \mu_{1}, b \mu_{2}), & \text{w.p. } \, p_{3}=\frac{1}{b}\end{cases}
	\end{equation}		
	
	(ii) if $a>c$, 
	\begin{equation}
	\label{distr:zeroline2}
	\begin{cases}(0,0), & \text{w.p. } \, p_{1}=\frac{(a-1)(b-1)-(c-1)^2}{(a b-c^{2})} \\ (0, \frac{a b-c^{2}}{a-c} \mu_{2}), & \text{w.p. } \,  p_{2}=\frac{(a-c)^{2}}{\left(a b-c^{2}\right) a} \\ (a \mu_{1}, c \mu_{2}), & \text{w.p. } \, p_{3}=\frac{1}{a}.\end{cases}
	\end{equation}		
	For both cases, the optimal values are all $ v_P(q;\theta)=\mu_1 + \mu_2 - q\cdot \frac{a+b - 2c}{ab - c^2}$.
\end{lemma}

\begin{proof}[Proof of Lemma \ref{lem:1}]	 



Recall the inequality $ab-c^2 \geq 0$ in \eqref{fz2}. We now further assume $b>c$. 

First, we shall verify the feasibility of the primal solution \eqref{distr:zeroline1}. Due to $b>c$ and $(a-1)(b-1)-(c-1)^2\geq0$, we can verify the non-negativity of the primal supporting points and $p_1, p_2, p_3>0$. Besides, it is easy to verify that  $p_1+p_2+p_3=1$. In addition, we can confirm the following moment constraints hold, $\Ex_{\mathbb{P}}[X_1]=\mu_1, \Ex_{\mathbb{P}}[X_2]=\mu_2,\Ex_{\mathbb{P}}[X_1^2]=a\mu_1^2, \Ex_{\mathbb{P}}[X_1 X_2]=c\mu_1\mu_2, \Ex_{\mathbb{P}}[X_2^2]=b\mu_2^2$. In all,  the primal solution \eqref{distr:zeroline1} is feasible for problem \eqref{eq:bivariatePrimal_app}.
	
Second, we shall verify the dual feasibility of the following dual solution $\boldsymbol{z}^*$ with
 \begin{equation}\label{dualcond1}
	\boldsymbol{z}^*=\left(0,1-\frac{2q(b-c)}{\mu_1(ab-c^2)},1-\frac{2q(a-c)}{\mu_2(ab-c^2)},\frac{q(b-c)^{2}}{\left(a b-c^{2}\right)^{2} \mu_{1}^{2}}, \frac{q(a-c)^{2}}{\left(a b-c^{2}\right)^{2} \mu_{2}^{2}}, \frac{2 q(a-c)(b-c)}{\left(a b-c^{2}\right)^2 \mu_{1} \mu_{2}}\right).
 \end{equation}
	We consider two dual constraint functions $h_1(\boldsymbol{x};\boldsymbol{z}^*)$ and $h_2(\boldsymbol{x};\boldsymbol{z}^*)$. Since $z_4^*\geq 0$ and $4z_4^*z_5^*-{z_6^*}^2=0$, then the Hessian matrix $\begin{pmatrix}
	2z_4^* & z_6^* \\
	z_6^* & 2z_5^*
	\end{pmatrix}$ is positive semidefinite, which means that both functions are convex. 
 Next, we take the derivative of functions $h_1(\boldsymbol{x};\boldsymbol{z}^*)$ and $h_2(\boldsymbol{x};\boldsymbol{z}^*)$,
\begin{equation}\label{h1}
    \nabla_{\boldsymbol{x} }h_1(\boldsymbol{x};\boldsymbol{z}^*)=(z_2^*+2z_4^*x_1+z_6^*x_2,\; z_3^*+2z_5^*x_2+z_6^*x_1)^T
\end{equation}
and 
\begin{equation}\label{h2}
    \nabla_{\boldsymbol{x} }h_2(\boldsymbol{x};\boldsymbol{z}^*) =(z_2^*-1+2z_4^*x_1+z_6^*x_2, \;z_3^*-1+2z_5^*x_2+z_6^*x_1)^T.
\end{equation}
Define $\boldsymbol{x}={(x_1,x_2)}^T$, $\boldsymbol{x}^*=(0,0)^T$,
 $\overline{\boldsymbol{x}}^*=(\frac{ab-c^2}{b-c}\mu_1,0)^T$. 
Substitute $x^*$ into equation \eqref{h1}, and we have
$$\nabla_{\boldsymbol{x}}h_1(\boldsymbol{x}^*;\boldsymbol{z}^*) ^T(\boldsymbol{x}-\boldsymbol{x}^*)=\frac{(b-1)(Q_b^2-q^2)}{(ab-c^2)\mu_1^2}x_1+\frac{(a-1)(Q_a^2-q^2)}{(ab-c^2)\mu_2^2}x_2.$$
Since $a>1$ and $b>1$ in Assumption 1, $ Q_a \geq q, Q_b \geq q$ in condition 1, and inequality \eqref{fz2}, then we have
$$\nabla_{\boldsymbol{x} }h_1(\boldsymbol{x}^*;\boldsymbol{z}^*) ^T(\boldsymbol{x}-\boldsymbol{x}^*)\geq 0, \;\; \forall \boldsymbol{x}\in \R_+^2.$$
Note that $h_1(\boldsymbol{x};\boldsymbol{z}^*)$ is a convex function, which implies for any $ \boldsymbol{x}\in \R_+^2 $, 
$$h_1(\boldsymbol{x};\boldsymbol{z}^*)\geq h_1(\boldsymbol{x}^*;\boldsymbol{z}^*)+ \nabla_{\boldsymbol{x} }h_1(\boldsymbol{x}^*;\boldsymbol{z}^*) ^T(\boldsymbol{x}-\boldsymbol{x}^*)\geq h_1(\boldsymbol{x}^*;\boldsymbol{z}^*)=0.$$
For function $h_2(\boldsymbol{x};\boldsymbol{z}^*)$, substitute $\overline{\boldsymbol{x}}^*$ into equation \eqref{h2}, and we have 
\begin{equation*}
    \nabla_{\boldsymbol{x} }h_2(\overline{\boldsymbol{x}}^*;\boldsymbol{z}^*)={\begin{pmatrix}
	-\frac{2q(b-c)}{\mu_1(ab-c^2)}+\frac{2q(b-c)^2}{(ab-c^2)^2\mu_1^2}\frac{ab-c^2}{b-c}\mu_1  \\
	-\frac{2q(a-c)}{\mu_2(ab-c^2)}+\frac{2q(a-c)(b-c)}{(ab-c^2)^2\mu_1\mu_2}\frac{ab-c^2}{b-c}\mu_1
	\end{pmatrix}}={\begin{pmatrix}
	    0\\0
	\end{pmatrix}}.
\end{equation*}
Note that the function $h_2(\boldsymbol{x};\boldsymbol{z}^*)$ is convex, which implies for any $ \boldsymbol{x}\in \R_+^2 $, 
$$h_2(\boldsymbol{x};\boldsymbol{z}^*)\geq h_2(\overline{\boldsymbol{x}}^*;\boldsymbol{z}^*)+ \nabla_{\boldsymbol{x} }h_2(\overline{\boldsymbol{x}}^*;\boldsymbol{z}^*) ^T(\boldsymbol{x}-\overline{\boldsymbol{x}}^*)= h_2(\overline{\boldsymbol{x}}^*;\boldsymbol{z}^*)=0.$$ 
In all, we prove that the dual solution $\boldsymbol{z}^*$ is feasible for problem \eqref{eq:dual}.

Lastly, it is easy to verify the objective values for the feasible primal-dual pair are all equal to $\mu_1 + \mu_2 - q\cdot \frac{a+b - 2c}{ab - c^2}$. Hence, the zero duality gap implies the optimality of the feasible primal-dual pair.

When $a>c$, the dual solution is the same as that of $b>c$ in equation \eqref{dualcond1}. By the similar analysis of the $b>c$ case, we can verify the distribution \eqref{distr:zeroline2} is an optimal solution. The last case is when $a=b=c$, in which $X_1$ and $X_2$ are perfectly correlated and Scarf's bound in Lemma 1 will solve the problem.
\end{proof}	

\begin{lemma}
	\label{lem:2}
	Suppose that $Q_b < q, Q_b \leq \zeta_b$ and Assumption 1 hold. Then an optimal distribution for primal distribution for problem
	\ref{eq:bivariatePrimal_app}
	can be characterized as
 \begin{equation}
     \label{distr:pointline2}
     \begin{cases}
	(q-Q_b,0)  & \text{w.p. } \, p_1=\frac{b-1}{2b}+\frac{q(b-1)+\mu_1(c-b)}{2bQ_b} \\
	(q+Q_b,0) & \text{w.p. }  \, p_2=\frac{b-1}{2b}-\frac{q(b-1)+\mu_1(c-b)}{2bQ_b}\\
	(c\mu_1, b\mu_2) & \text{w.p. } \, p_3=\frac{1}{b}\\
	\end{cases},
 \end{equation}
	where $Q_b=\sqrt{q^2 + 2q\cdot \frac{c-b}{b-1}\mu_1 + \frac{ab-c^2}{b-1}\mu_1^2}$. The optimal value is $v_P(q;\theta)=\frac{b-1}{2b}( (q+Q_b) - \frac{b-c}{b-1}\mu_1 ) + \mu_1+\mu_2 - q$.
\end{lemma}

\begin{proof}[Proof of Lemma \ref{lem:2}]



	First, we shall verify the feasibility of the primal solution \eqref{distr:pointline2}. 
Since $(a-1)(b-1)\geq (c-1)^2$ in Assumption 1, we   have $(ab-c^2)(b-1)-(b-c)^2 \geq (a+b-2c)(b-1)-(b-c)^2=(a-1)(b-1)-(c-1)^2 \geq 0$, and it further leads to
	$$Q_b^2 - (\frac{b-c}{b-1}\mu_1-q)^2=\frac{(ab-c^2)(b-1)-(b-c)^2}{(b-1)^2}\mu_1^2\geq0.$$ 
With $Q_b\geq 0$, we get $Q_b \geq \frac{b-c}{b-1}\mu_1-q$ and $Q_b \geq q-\frac{b-c}{b-1}\mu_1$. 
Now we rewrite $p_1$ and $p_2$ to be 
$$p_1=\frac{b-1}{2bQ_b}(Q_b+q-\mu_1\frac{b-c}{b-1})\;\text{ and }\;p_2=\frac{b-1}{2bQ_b}(Q_b-q+\mu_1\frac{b-c}{b-1}).$$
Thus,  $p_1 \geq 0$ is due to $Q_b \geq \frac{b-c}{b-1}\mu_1-q$, and $p_2 \geq 0$ is due to $Q_b \geq q-\frac{b-c}{b-1}\mu_1$. Besides, $p_3\geq 0$ holds obviously. 
Furthermore, by straightforward calculations, we can also 
 verify that the following constraints hold, $\Ex_{\mathbb{P}}[1]=p_1+p_2+p_3=1$,
 $\Ex_{\mathbb{P}}[X_1]=\mu_1, \Ex_{\mathbb{P}}[X_2]=\mu_2,\Ex_{\mathbb{P}}[X_1^2]=a\mu_1^2, \Ex_{\mathbb{P}}[X_1 X_2]=c\mu_1\mu_2, \Ex_{\mathbb{P}}[X_2^2]=b\mu_2^2$. In all, the primal solution \eqref{distr:pointline2} is feasible for problem \eqref{eq:bivariatePrimal_app}.
	
	Second, we shall verify the dual feasibility of the following dual solution $$\boldsymbol{z}^*=\left(\frac{(Q_b-q)^2}{4Q_b},\frac{Q_b-q}{2Q_b},1-\frac{Q_b+q}{2Q_b}\frac{Q_b+q-c\mu_1}{b\mu_2},\frac{1}{4Q_b}, \frac{(Q_b+q-c\mu_1)^2}{4b^2\mu_2^2Q_b},\frac{Q_b+q-c\mu_1}{2b\mu_2 Q_b}\right).$$
	We consider two dual constraint functions $h_1(\boldsymbol{x};\boldsymbol{z}^*)$ and $h_2(\boldsymbol{x};\boldsymbol{z}^*)$. 
Since $z_4^*\geq 0$ and $4z_4^*z_5^*-{z_6^*}^2=0$, then the Hessian matrix $\begin{pmatrix}
	2z_4^* & z_6^* \\
	z_6^* & 2z_5^*
	\end{pmatrix}$ is positive semidefinite, which means that both functions are convex. Next, we define $\boldsymbol{x}={(x_1,x_2)}^T$, $\boldsymbol{x}^*=(q-Q_b,0)^T$ and $\overline{\boldsymbol{x}}^*=(q+Q_b,0)^T$.
Then, we have 
$$\nabla_{\boldsymbol{x}}h_1(\boldsymbol{x}^*;\boldsymbol{z}^*) ^T(\boldsymbol{x}-\boldsymbol{x}^*)
=\frac{\zeta_b-Q_b}{b\mu_2}x_2 \geq 0, \;\;\forall \boldsymbol{x}\in \R_+^2.$$
Since $h_1(\boldsymbol{x};\boldsymbol{z}^*)$ is a convex function, for any $\boldsymbol{x}\in \R_+^2$, we have 
$h_1(\boldsymbol{x};\boldsymbol{z}^*)\geq h_1(\boldsymbol{x}^*;\boldsymbol{z}^*)=0.$
For function $h_2(\boldsymbol{x};\boldsymbol{z}^*)$, substitute $\overline{\boldsymbol{x}}^*$ into equation \eqref{h2}, and we have $\nabla_x h_2(\overline{\boldsymbol{x}}^*;\boldsymbol{z}^*)=(0,0)^T$. Thus $h_2(\boldsymbol{x};\boldsymbol{z}^*) \geq h_2(\overline{\boldsymbol{x}}^*;\boldsymbol{z}^*)=0$. In all, we prove that the dual solution $\boldsymbol{z}^*$ is feasible for problem \eqref{eq:dual}.


Lastly, it is easy to verify the objective values for the feasible primal-dual pair are all equal to 

\noindent $\frac{b-1}{2b}\left( (q+Q_b) - \frac{b-c}{b-1}\mu_1 \right) + \mu_1+\mu_2 - q$. Hence, the zero duality gap implies the optimality of the feasible primal-dual pair.
\end{proof} 

\begin{lemma}
	\label{lem:3}
	Suppose that 
$Q_a < q, Q_a \leq \zeta_a$ 
 and Assumption 1 hold. Then an optimal distribution for primal distribution for problem
	\ref{eq:bivariatePrimal_app}
	can be characterized as
 \begin{equation}\label{distr:pointline3}
  \begin{cases}
	(0,q-Q_a)  & \text{w.p. } \, p_1=\frac{a-1}{2a}+\frac{q(a-1)+\mu_2(c-a)}{2aQ_a} \\
	(0,q+Q_a) & \text{w.p. }  \, p_2=\frac{a-1}{2a}-\frac{q(a-1)+\mu_2(c-a)}{2aQ_a}\\
	( a\mu_1,c\mu_2) & \text{w.p. } \, p_3=\frac{1}{a}\\
	\end{cases},
 \end{equation}
 where $Q_a=\sqrt{q^2 + 2q\cdot \frac{c-a}{a-1}\mu_2 + \frac{ab-c^2}{a-1}\mu_2^2}$.
The optimal value is  $v_P(q;\theta)=
 \frac{a-1}{2a}\left( (q+Q_a) - \frac{a-c}{2a}\mu_2 \right) + \mu_1+\mu_2 - q$.
\end{lemma}

\begin{proof}[Proof of Lemma \ref{lem:3}]
The proof is similar to that in Lemma \ref{lem:2}. Note that the conditions in Lemma \ref{lem:2} and \ref{lem:3} are symmetric, in some sense, through replacing $b$ by $a$ and $\mu_2$ by $\mu_1$.  
\end{proof}


\begin{lemma}
	\label{lem:4}
	Suppose that 
 $Q_b < q, Q_b \leq -\zeta_b$
 and Assumption 1 hold. Then an optimal distribution for primal distribution for problem
	\ref{eq:bivariatePrimal_app}
	can be characterized as
\begin{equation}\label{distr:pointline4}
 \begin{cases}
	(q-Q_b,0)  & \text{w.p. } \, p_1=\frac{b-1}{2b}+\frac{q(b-1)+\mu_1(c-b)}{2bQ_b} \\
	(q+Q_b,0) & \text{w.p. }  \, p_2=\frac{b-1}{2b}-\frac{q(b-1)+\mu_1(c-b)}{2bQ_b}\\
	(c\mu_1, b\mu_2) & \text{w.p. } \, p_3=\frac{1}{b}\\
	\end{cases},
 \end{equation}
 where $Q_b=\sqrt{q^2 + 2q\cdot \frac{c-b}{b-1}\mu_1 + \frac{ab-c^2}{b-1}\mu_1^2}$.
	The optimal value is $v_P(q;\theta)=
 \frac{b-1}{2b} \left(\frac{b-c}{b-1}\mu_1 - (q-Q_b) \right)$.
\end{lemma}

\begin{proof}[Proof of Lemma \ref{lem:4}]



	First, we observe that the optimal primal solution is the same as that in Lemma \ref{lem:2}, so the primal feasibility of this solution holds.

Second,  we shall verify the dual feasibility of the following dual solution$$\boldsymbol{z}^*=\left(\frac{(Q_b-q)^2}{4Q_b},\frac{Q_b-q}{2Q_b},\frac{(Q_b-q)(q-Q_b-c\mu_1)}{2bQ_b\mu_2},\frac{1}{4Q_b},\frac{(q-Q_b-c\mu_1)^2}{4Q_bb^2},\frac{q-Q_b-c\mu_1}{2bQ_b\mu_2}\right).$$
We consider two dual constraint functions $h_1(\boldsymbol{x};\boldsymbol{z}^*)$ and $h_2(\boldsymbol{x};\boldsymbol{z}^*)$. 
Note that $z_4^*\geq 0$ and $4z_4^*z_5^*-{z_6^*}^2=0$, then the Hessian matrix $\begin{pmatrix}
	2z_4^* & z_6^* \\
	z_6^* & 2z_5^*
	\end{pmatrix}$ is positive semidefinite, which means that both functions are convex. Next, we construct $\boldsymbol{x}^*=(q-Q_b,0)^T$ and $\overline{\boldsymbol{x}}^*=(q+Q_b,0)^T$.
 Then we have 
$$\nabla_{\boldsymbol{x}}h_1(\boldsymbol{x}^*;\boldsymbol{z}^*) ^T(\boldsymbol{x}-\boldsymbol{x}^*)
= \frac{-Q_b-\zeta_b}{b\mu_2}x_2 \geq 0, \;\;\forall \boldsymbol{x}\in \R_+^2.$$
Since $h_1(\boldsymbol{x};\boldsymbol{z}^*)$ is a convex function, for any $\boldsymbol{x}\in \R_+^2$, we have 
$h_1(\boldsymbol{x};\boldsymbol{z}^*)\geq h_1(\boldsymbol{x}^*;\boldsymbol{z}^*)=0.$
For function $h_2(\boldsymbol{x};\boldsymbol{z}^*)$, substitute $\overline{\boldsymbol{x}}^*$ into equation \eqref{h2}, and we have $\nabla_x h_2(\overline{\boldsymbol{x}}^*;\boldsymbol{z}^*)=(0,\;0)^T$. Thus, $h_2(\boldsymbol{x};\boldsymbol{z}^*) \geq h_2(\overline{\boldsymbol{x}}^*;\boldsymbol{z}^*)=0$. In all, $h_1(\boldsymbol{x};\boldsymbol{z}^*)\geq0$ and 
$h_2(\boldsymbol{x};\boldsymbol{z}^*)\geq0$ hold for all $\boldsymbol{x}\in \R_+^2$.

Lastly, it is easy to verify the objective values for the feasible primal-dual pair equal to $\frac{b-1}{2b} \left(\frac{b-c}{b-1}\mu_1 - (q-Q_b) \right)$. Hence, the zero duality gap implies the optimality of the feasible primal-dual pair. \end{proof}

\begin{lemma}
	\label{lem:5}
	Suppose that $2(a-c)q > (ab-c^2)\mu_2,\ Q_a  \leq -c\mu_2 - a\mu_1+q,$ and Assumption 1 hold. Then an optimal distribution for primal distribution for problem
	\ref{eq:bivariatePrimal_app} 
	can be characterized as
 \begin{equation}\label{distr:pointline5}
	 \begin{cases}
	(0,q-Q_a)  & \text{w.p. } \, p_1=\frac{a-1}{2a}+\frac{q(a-1)+\mu_2(c-a)}{2aQ_a} \\
	(0,q+Q_a) & \text{w.p. }  \, p_2=\frac{a-1}{2a}-\frac{q(a-1)+\mu_2(c-a)}{2aQ_a}\\
	( a\mu_1,c\mu_2) & \text{w.p. } \, p_3=\frac{1}{a}\\
	\end{cases},
 \end{equation}
 where $Q_a=\sqrt{q^2 + 2q\cdot \frac{c-a}{a-1}\mu_2 + \frac{ab-c^2}{a-1}\mu_2^2}$.
	The optimal value is $v_P(q;\theta)=
 \frac{a-1}{2a}\left(\frac{a-c}{a-1}\mu_2 -(q-Q_a) \right)$.

\end{lemma}

\begin{proof}[Proof of Lemma \ref{lem:5}]
The proof is similar to that in Lemma \ref{lem:4}. Note that the conditions in Lemma \ref{lem:4} and \ref{lem:5} are symmetric, in some sense, through	replacing $b$ by $a$ and $\mu_2$ by $\mu_1$.  
\end{proof}

\begin{lemma}
	\label{lem:6}
	Suppose that 
  $Q_a > |\zeta_a|, Q_b > |\zeta_b|$
 and Assumption 1 hold. Then an optimal distribution of problem
	\ref{eq:bivariatePrimal_app}
	 can be characterized as
	\begin{equation} \label{eq: dist6}
	\begin{cases}
	x^{(1)}=((1-t_0)\frac{U_a}{U_c}\mu_1,t_1)  & \text{w.p. } \, p_1=\frac{U_b\mu_2}{2Q_c t_1} \\
	x^{(2)}=(q-Q_c,0)  & \text{w.p. } \, p_2=\frac{U_a  \mu_1}{2Q_c t_1}t_0 \\
	x^{(3)}=(0,q+Q_c) & \text{w.p. }  \, p_3=\frac{V_b \,u_2}{2Q_c t_2}t_0 \\
	x^{(4)}=(t_2,(1-t_0)\frac{V_b}{V_c}\mu_2) & \text{w.p. }  \, p_4=\frac{V_a\mu_1}{2Q_ct_2}\\
	\end{cases} \;\; 
	\end{equation}
	where  
$	U_a=(q+Q_c)-(a\mu_1+c\mu_2), V_a=(a\mu_1+c\mu_2)-(q-Q_c), U_b=(q+Q_c)-(b\mu_2+c\mu_1), V_b=(b\mu_2+c\mu_1)-(q-Q_c), U_c=(q+Q_c)-(\mu_1+\mu_2), V_c=(\mu_1+\mu_2)-(q-Q_c)$, $Q_c = \sqrt{q^2 - 2q(\mu_1 + \mu_2) + a\mu_1^2 + b\mu_2^2 + 2c\mu_1\mu_2}$, and
$$	t_1= \frac{U_b}{U_c}\mu_2+t_0\frac{U_a}{U_c}\mu_1, \;\;\;
	t_2= \frac{V_a}{V_c}\mu_1+t_0\frac{V_b}{V_c}\mu_2, \;\;\; t_0=\frac{Det(M)}{Det(M)+\Sigma_{12} U_c V_c}\;, $$ 
 with covariance matrix $M$ defined in (\ref{eq: M_apx}).
 The optimal value is $v_P(q;\theta)=\frac{1}{2}V_c$.
\end{lemma}

\begin{proof}[Proof of Lemma \ref{lem:6}]
	

 First, we shall verify the feasibility of the primal solution (\ref{eq: dist6}). Because of 
 the condition 6 with $Q_a>|q-a\mu_1 - c\mu_2|, Q_b> |q-c\mu_1 - b\mu_2|$, we have $V_a, V_b, U_a, U_b>0$ immediately. Moreover, we also have $V_c, U_c>0$ due to $Q_c^2-(q-\mu_1-\mu_2)^2=(a-1)\mu_1^2+2(c-1)\mu_1\mu_2+(b-1)\mu_2^2=\V ar(X_1+X_2)>0$. Thus, the distribution (\ref{eq: dist6}) has positive support and probability. As $$(p_1+p_2)+(p_3+p_4)=
	\frac{U_b\mu_2+t_0U_a\mu_1}{2Q_c t_1}+\frac{t_0V_b\mu_2+V_a\mu_1}{2Q_ct_2}=\frac{U_c}{2Q_c}+\frac{V_c}{2Q_c}=1,$$
	we conclude the primal solution (\ref{eq: dist6}) is a well-defined distribution. To verify this distribution satisfies corresponding constraints, we state two useful equations $q-Q_c=\frac{U_a}{U_c}\mu_1+\frac{U_b}{U_c}\mu_2, q+Q_c=\frac{V_b}{V_c}\mu_2+\frac{V_a}{V_c}\mu_1$
	that will be frequently employed. Specially, we have
	\begin{align*}
	\Ex[X_1]&=\sum_{i=1}^4 p_i x^{(i)}_1= \mu_1\Big(\frac{(1-t_0)U_aU_b\mu_2+t_0U_a^2\mu_1+t_0U_aU_b\mu_2}{2Q_cU_ct_1}+\frac{V_a}{2Q_c}\Big)=\mu_1(\frac{U_a}{2Q_c}+\frac{V_a}{2Q_c})=\mu_1,    \\
	\Ex[X_2]&=\sum_{i=1}^4 p_i x^{(i)}_2=\mu_2 \Big(\frac{V_b}{2Q_c}+\frac{t_0V_b^2\mu_2+t_0V_aV_b\mu_1+(1-t_0)V_aV_b\mu_1}{2Q_cV_ct_2}\Big)=\mu_2(\frac{U_b}{2Q_c}+\frac{V_b}{2Q_c})=\mu_2,\\
	\Ex[X_1X_2]&=\sum_{i=1}^4 p_i x^{(i)}_1  x^{(i)}_2 =(1-t_0)\frac{\mu_1\mu_2}{2Q_c}\Big(\frac{U_a U_b}{U_c}+\frac{V_aV_b}{V_c}\Big)=(1-t_0)\frac{\mu_1\mu_2}{2Q_c}\Big(\frac{2Q_c}{U_cV_c}\frac{Det(M)}{\mu_1\mu_2}+2cQ_c\Big)=c\mu_1\mu_2, 
	\end{align*}
 \begin{align*}
	\Ex[X_1(X_1+X_2)] &=\sum_{i=1}^4 p_i x^{(i)}_1 ( x^{(i)}_1 +x^{(i)}_2)=(q-Q_c)[p_1(q-Q_c-t_1))+p_2(q-Q_c)]+p_4(q+Q_c)t_2,\\
	&=(q-Q_c)(\mu_1-p_3x^{(3)}_1-p_4x^{(4)}_1)+p_4(q+Q_c)t_2=(q-Q_c)\mu_1+2Q_cp_4t_2=a\mu_1^2+c\mu_1\mu_2\\
	\Ex[(X_1+X_2)X_2] &=\sum_{i=1}^4 p_i ( x^{(i)}_1 +x^{(i)}_2)x^{(i)}_2 =p_1(q-Q_c)t_1+(q+Q_c)[p_3(q+Q_c)+p_4(q+Q_c-t_2)],\\
	&=p_1(q-Q_c)t_1+(q+Q_c)(\mu_2-p_1x^{(1)}_2-p_2x^{(2)}_2)=(q+Q_c)\mu_2-2Q_cp_1t_1=b\mu_2^2+c\mu_1\mu_2.
	\end{align*}
	In all, the primal feasibility holds.
	
	Second, we shall verify the dual feasibility of the following dual solution $$\boldsymbol{z}^*=\left(\frac{(Q_b-q)^2}{4Q_c},\frac{Q_c-q}{2Q_c},\frac{Q_c-q}{2Q_c},\frac{1}{4Q_C}, \frac{1}{4Q_C},\frac{1}{2Q_C}\right).$$
	 By straightforward calculations, we have $h_1(x;\boldsymbol{z}^*)=\frac{1}{4Q_c}(x_1+x_2-q+Q_c)^2 \geq 0$ and $h_2(x;\boldsymbol{z}^*)=\frac{1}{4Q_c}(x_1+x_2-q-Q_c)^2 \geq 0$. Thus, $\boldsymbol{z}^*$ is a dual feasible solution.
	
	Lastly, it is easy to verify the objective values for the feasible primal-dual pair equal to $\frac{1}{2}V_c$. Hence, the zero duality gap implies the optimality of the feasible primal-dual pair. \end{proof}

\section*{Appendix C: Proof of Proposition 2}
\setcounter{pro}{4}

From Theorem 2.1 in \cite{shapiro2002minimax}, the bivariate moment problem
\setcounter{equation}{0}
\renewcommand{\theequation}{C.\arabic{equation}}
\begin{equation}
	\label{eq:multiPieceBivariate}
	\begin{split}
 \sup_{\mathbb{P}\in \mathcal{F}(\theta)}\ \mathbb{E}_{\mathbb{P}} \left[\max_{k=1, ..., K} \left\{ 
 \ell_{k}(\boldsymbol{X}) \right\} \right]
	\end{split}
\end{equation}
is equivalent to its dual problem
\begin{align*}
&\inf_{\boldsymbol{z}}\quad z_1 + \mu_1 z_2 + \mu_2 z_3 + \Sigma_{11} z_4 + \Sigma_{22}z_5 + \Sigma_{12} z_6 \\
\text{s.t.}\quad z_1 + z_2 x_1+ z_3 x_2+ &z_4 x_1^2+ z_5 x_2^2+ z_6 x_1x_2\geq \max_{k=1, ..., m} \left\{ w_{1k} + w_{2k} x_1 + w_{3k} x_2 + w_{4k} x_1^2 + w_{5k} x_2^2 + w_{6k} x_1x_2 \right\}, 
\end{align*}
for all $\boldsymbol{x}\in \mathbb{R}_+^2$.
Let $y_1^2 = x_1$ and $y_2^2 = x_2$, and then we have
\begin{equation*}
\begin{split}
&z_1 + z_2 x_1+ z_3 x_2+ z_4 x_1^2+ z_5 x_2^2+ z_6 x_1x_2 \geq w_{1k} + w_{2k} x_1 + w_{3k} x_2 + w_{4k} x_1^2 + w_{5k} x_2^2 + w_{6k} x_1x_2,\;\; \forall \boldsymbol{x}\in \mathbb{R}_+^2 \\
\Leftrightarrow\quad &z_1 - w_{1k} + (z_2 - w_{2k}) y_1^2+ (z_3 - w_{3k}) y_2^2+ (z_4 - w_{4k}) y_1^4 + (z_5 - w_{5k}) y_2^4+ (z_6 - w_{6k}) y_1^2y_2^2\geq 0,\;\; \forall \boldsymbol{y}\in \mathbb{R}^2.
\end{split}
\end{equation*}
The left-hand-side of the inequality can be represented by the sum of squares of polynomials, that is, the above inequality is equivalent to
\begin{align*}
&\begin{pmatrix}
y_1^2 \\
y_1y_2 \\
y_2^2 \\
y_1 \\
y_2 \\
1
\end{pmatrix}^T
M \left(\boldsymbol{z} - \boldsymbol{w}^{(k)}, \boldsymbol{g}^{(k)}, \boldsymbol{h}^{(k)} \right) \begin{pmatrix}
y_1^2 \\
y_1y_2 \\
y_2^2 \\
y_1 \\
y_2 \\
1
\end{pmatrix} \geq 0,\qquad \forall \boldsymbol{y}\in \mathbb{R}^2 \\
\Leftrightarrow\quad &M \left(\boldsymbol{z} - \boldsymbol{w}^{(k)}, \boldsymbol{g}^{(k)}, \boldsymbol{h}^{(k)} \right) \succeq 0
\end{align*}
for some $\boldsymbol{g}^{(k)}$ and $\boldsymbol{h}^{(k)}$. Therefore, the original problem \eqref{eq:multiPieceBivariate} can be reformulated by
\begin{equation*}
\begin{split}
\inf_{\boldsymbol{z},G,H}\quad &z_1 + \mu_1 z_2 + \mu_2 z_3 + \Sigma_{11} z_4 +  \Sigma_{22} z_5 + \Sigma_{12} z_6 \\
\text{s.t.}\quad &M \left(\boldsymbol{z} - \boldsymbol{w}^{(k)}, \boldsymbol{g}^{(k)}, \boldsymbol{h}^{(k)} \right) \succeq 0,\qquad k=1, ..., m
\end{split}
\end{equation*}
with $\boldsymbol{z} \in  \mathbb{R}^{6}, {G} \in \mathbb{R}^{3 \times K}, {H} \in \mathbb{R}^{3 \times K}$ and the matrix  $M(\tilde{\boldsymbol{z}},\boldsymbol{g},{\boldsymbol{h}})$ defined as follows:
\begin{equation*}
\begin{split}
	 \begin{pmatrix}
	\tilde{z}_4 & 0 & -{g}_1 & 0 & -{h}_1 & -{g}_2 \\
	0 & \tilde{z}_6 +2{g}_1 & 0 & {h}_1 & -{h}_2 & -h_3 \\
	-{g}_1 & 0 & \tilde{z}_5 & {h}_2 & 0 & -{g}_3 \\
	0 & {h}_1 & {h}_2 & \tilde{z}_2 + 2{g}_2 & {h}_3 & 0 \\
	-{h}_1 & -{h}_2 & 0 & {h}_3 & \tilde{z}_3 + 2{g}_3 & 0 \\
	-{g}_2 & -{h}_3 & -{g}_3 & 0 & 0 & \tilde{z}_1 
	\end{pmatrix}
\end{split}
\end{equation*}
where $\boldsymbol{w}^{(k)}, \boldsymbol{g}^{(k)}, \boldsymbol{h}^{(k)}$ are the $k$th column vectors of matrices $W$, ${G}$, ${H}$ respectively. 

\section*{Appendix D: Proof of Proposition 3}

From Theorem 4.5.7 in book \cite{casella2002statistical}, we know that $|\rho|=1$ if and only if there exists numbers $\varphi_1\neq 0$ and $\varphi_2$ such that $\mathbb{P}(X_1=\varphi_1X_2+\varphi_2)=1$. 
If $X_1$ and $X_2$ are perfectly correlated with $\rho=1$, we have the following moment constraints due to $X_1=\varphi_1X_2+\varphi_2$ and $\varphi_1>0$:
 $$\mathbb{E} [\varphi_1X_2+\varphi_2] = \mu_1, \mathbb{E}[X_2] = \mu_2, \mathbb{E}[(\varphi_1X_2+\varphi_2)^2] = a\mu_1^2, \mathbb{E} [X_2^2]=b\mu_2^2.$$
From these equations, $\varphi_1$ and $\varphi_2$ should satisfy $\varphi_1\mu_2+\varphi_2=\mu_1$ and $b\varphi_1^2\mu_2^2+\varphi_2^2+2\varphi_1\varphi_2\mu_2=a\mu_1^2.$
Then we can obtain $\varphi_1=\sqrt{\frac{a-1}{b-1}}\frac{\mu_1}{\mu_2}$ and $\varphi_2=\left(1-\sqrt{\frac{a-1}{b-1}}\right) \mu_1$ and thus $X_1=\sqrt{\frac{a-1}{b-1}}\frac{\mu_1}{\mu_2} X_2+\left(1-\sqrt{\frac{a-1}{b-1}}\right) \mu_1$.
Moreover, note that we assume $1\leq a\leq b$, then $\left(1-\sqrt{\frac{a-1}{b-1}}\right) \mu_1\geq0$. From the nonnegativity of $X_2$, we know that $X_1\geq\left(1-\sqrt{\frac{a-1}{b-1}}\right) \mu_1$.

\section*{Appendix E: Closed-Form Solution of the DRO Newsvendor Problem}
We study the DRO newsvendor problem as follows:
\begin{equation*}
\inf_{q\geq 0} \left\{\sup_{\mathbb{P}\in \mathcal{F}(\theta)}\ \mathbb{E}_{\mathbb{P}} \left[\left(X_1+X_2 -q \right)_+\right] + (1-\eta)q \right\},
 \end{equation*} 
where $\mathcal{F}(\theta)$ is the mean-covariance ambiguity set in (\ref{eq: amSet_app}), and the critical ratio $\eta$ is a given constant in $(0,1)$. 

We explain the procedure of solving the closed-form solution $q^*$ introduced in Section 4.1. Firstly, each $q_i^*$ locates at either a stationary point of $v_P(q;\theta) + (1- \eta)q$ or a boundary point of the interval $A_i$. The formulation of each interval $A_i$ is discussed in Table \ref{tab:tableRegionQ}. Thus, we study the stationary point of 
\begin{equation*}
    f(q) = v_P(q;\theta) + (1- \eta)q
\end{equation*}
in each interval $A_i$. For simplicity, we denote
\begin{align*}
    S_a(\eta) &= \sqrt{\frac{(a-1)(ab-c^2) - (c-a)^2}{4a\eta (a-a\eta -1)}}, \\
    S_b(\eta) &= \sqrt{\frac{(b-1)(ab-c^2) - (c-b)^2}{4b\eta (b-b\eta -1)}}, \\
    S_c(\eta) &= \sqrt{\frac{(a-1)\mu_1^2 + (b-1)\mu_2^2 + 2(c-1)\mu_1\mu_2}{4\eta(1-\eta)}}.
\end{align*}


For condition 2 as an example, we have
\begin{equation*}
    \begin{split}
        f(q) &= \frac{b-1}{2b}\sqrt{q^2 - 2q \frac{b-c}{b-1}\mu_1 + \frac{ab-c^2}{b-1}\mu_1^2} + \left(\frac{b-1}{2b}-\eta \right)q + \frac{b+c}{2b}\mu_1 +\mu_2
    \end{split}
\end{equation*}
with its derivative
\begin{equation*}
    f'(q) = \frac{b-1}{2b} \frac{q - \frac{b-c}{b-1}\mu_1}{\sqrt{q^2 - 2q \frac{b-c}{b-1}\mu_1 + \frac{ab-c^2}{b-1}\mu_1^2}} + \frac{b-1}{2b}-\eta. 
\end{equation*}
Thus, the stationary point is obtained  as follows by setting  $f^{\prime}(q)=0$:
\begin{align*}
q &= \left[(2b\eta - b + 1) \sqrt{\frac{(b-1)(ab-c^2) - (c-b)^2}{4b\eta (b-b\eta -1)}} - (c-b) \right] \frac{\mu_1}{b-1} \\
&= \left[(2b\eta - b + 1) S_b(\eta) - (c-b) \right]\frac{\mu_1}{b-1},
\end{align*}
which exists only if $0 < \eta < 1- \frac{1}{b}$.

The stationary points under other conditions can be derived by the same approach shown under condition 2. We provide Table \ref{tab:stationaryPoints} to summarize the closed-form stationary points in all cases. Lastly, we remark that the stationary point shown in this table for each condition may be out of the interval $A_i$. In all, given this table, it is not difficult to obtain the closed-form expressions of the optimal solutions.

\setcounter{table}{0}
\renewcommand{\thetable}{E.\arabic{table}}
\begin{table}[htbp]
    \centering
    \begin{tabular}{c|cl}
    \hline
    Conditions & Stationary Points & Feasible $\eta$\\
    \hline
    condition 1 & $f(q)$ is linear in $q\in A_1$ & \\
    condition 2 & $\left[(2b\eta - b + 1) S_b(\eta) - (c-b) \right]\frac{\mu_1}{b-1}$, & $0 < \eta < 1- \frac{1}{b}$ \\
    condition 3 & $\left[(2a\eta - a + 1)S_a(\eta) - (c-a) \right]\frac{\mu_2}{a-1}$, & $0 < \eta < 1- \frac{1}{a}$ \\
    condition 4 & $\left[(2b\eta - b + 1) S_b(1-\eta) - (c-b) \right]\frac{\mu_1}{b-1}$, & $\frac{1}{b} < \eta < 1$ \\
    condition 5 & $\left[(2a\eta - a + 1)S_a(1-\eta) - (c-a) \right]\frac{\mu_2}{a-1}$, & $\frac{1}{a} < \eta < 1$ \\
    condition 6 & $(2\eta-1)S_c(\eta) + \mu_1 + \mu_2$, & $0 < \eta < 1$ \\
    \hline
    \end{tabular}
    \caption{Stationary Points}
    \label{tab:stationaryPoints}
\end{table}

\end{document}